\documentclass[reqno,a4paper,11pt]{amsart}
\usepackage[ngerman, english]{babel}
\usepackage[all,poly]{xy}
\usepackage[mathcal]{eucal}
\usepackage{enumerate}
\usepackage{amsmath,amssymb, amsthm, amsfonts}
\usepackage[pagebackref,colorlinks]{hyperref}
\usepackage[foot]{amsaddr}
\usepackage{graphicx}
\usepackage{mathabx}

\theoremstyle{plain}
\newtheorem {lemma}{Lemma}

\newtheorem {theorem}[lemma]{Theorem}
\newtheorem*{SCT}{SCT}
\newtheorem {corollary}[lemma]{Corollary}

\theoremstyle{definition}
\newtheorem{definition}[lemma]{Definition}

\newtheorem{remark}[lemma]{Remark}
\newtheorem {example}[lemma]{Example}

\DeclareMathOperator*{\plus}{\overset{.}{\bigplus}}
\newcommand{\+}{\overset{.}{+}}
\newcommand{\minus}{\overset{.}{-}}
\newcommand{\h}{\mathfrak{H}}

\title[Sandwich classification revisited]{Sandwich classification for $O_{2n+1}(R)$ and $U_{2n+1}(R,\Delta)$ revisited}
\author{Raimund Preusser}
\address{Department of Mathematics,
University of Brasilia, Brazil}
\email{raimund.preusser@gmx.de}

\begin{document}
\maketitle
\begin{abstract}
In a recent paper, the author proved that if $n\geq 3$ is a natural number, $R$ a commutative ring and $\sigma\in GL_n(R)$, then $t_{kl}(\sigma_{ij})$ where $i\neq j$ and $k\neq l$ can be expressed as a product of $8$ matrices of the form $^{\epsilon}\sigma^{\pm 1}$ where $\epsilon\in E_n(R)$. In this article we prove similar results for the odd-dimensional orthogonal groups $O_{2n+1}(R)$ and the odd-dimensional unitary groups $U_{2n+1}(R,\Delta)$ under the assumption that $R$ is commutative and $n\geq 3$. This yields new, short proofs of the Sandwich Classification Theorems for the groups $O_{2n+1}(R)$ and $U_{2n+1}(R,\Delta)$.
\end{abstract}
\section{Introduction}
The Sandwich Classification Theorem (SCT) for the lattice of subgroups of a linear group $G$ that are normalized by the elementary subgroup $E$ of $G$ is one of central points in the structure theory of linear groups. For general rings, not only semilocal or arithmetic ones, the SCT was first proved
by H. Bass \cite{3,4} for $GL_n$ under the stable range condition. This condition allowed
him to extract a transvection as well as to prove the standard commutator formula. The first proofs of the SCT for $G = GL_n$ over a commutative ring by J. Wilson \cite{20} and I. Golubchik \cite{6} used only direct calculations as neither localization techniques nor the standard commutator formula were known at that time. About ten years later the SCT was proved applying the standard commutator formula by L. Vaserstein \cite{11} and, independently, by Z. Borewich and N. Vavilov \cite{5}. In
the latter article the authors introduced a trick to stabilize a column of a matrix whereas Vaserstein used localization. The next generation of the SCT proofs works for Chevalley groups. L. Vaserstein \cite{12} and E. Abe \cite{1} used localization methods whereas N. Vavilov, E. Plotkin and A. Stepanov \cite{19} introduced the decomposition of unipotents, which in more detail was described in \cite{10} and further developed in \cite{15, 16, 18, 17, 13, 14} by Vavilov and his students and in \cite{7} by V. Petrov. For generalized hyperbolic unitary groups the SCT was announced by A. Bak and N. Vavilov in \cite{2}, but the proof has never been published. It appeared later in \cite{8} (which is essentially the author's thesis). Recently A. Stepanov \cite{9} proved the SCT for all Chevalley groups using the universal localization method.

Let $n\geq 3$ be a natural number and $R$ a commutative ring. The SCT for $GL_n(R)$ states the following (it is also true over almost commutative rings, cf. \cite{11}):
\begin{SCT}
Let $H$ be a subgroup of $GL_n(R)$. Then $H$ is normalized by $E_n(R)$ iff 
\begin{equation*}
E_n(R,I)\subseteq H\subseteq C_n(R,I)
\end{equation*}
for some ideal $I$ of $R$.
\end{SCT}
\noindent The ideal $I$ in the SCT is uniquely determined, namely $I=I(H)=\{x\in R\mid t_{12}(x)\in H\}$. Let now $\sigma\in GL_n(R)$ and set $H:={}^{E_n(R)}\sigma$, i.e. $H$ is the smallest subgroup of $GL_n(R)$ which contains $\sigma$ and is normalized by $E_n(R)$. Then, by the SCT, $H\subseteq C_n(R,I)$ where $I=I(H)$. It follows from the definition of $C_n(R,I)$ that $\sigma_{ij}, \sigma_{ii}-\sigma_{jj}\in I$ for any $i\neq j$. Hence, by the definition of $I(H)$, the matrices $t_{12}(\sigma_{ij})$ and $t_{12}(\sigma_{ii}-\sigma_{jj})$ can be expressed as products of matrices of the form $^{\epsilon}\sigma^{\pm 1}$ where $\epsilon\in E_n(R)$. In \cite{preusser_sct} the author showed how one can use the theme of the paper \cite{10} in order to find such expressions and gave boundaries for the number of factors. This yielded a new, very simple proof of the SCT for $GL_n(R)$. Further the author obtained similar results for the even-dimensional orthogonal groups $O_{2n}(R)$ and the even-dimensional unitary groups $U_{2n}(R,\Lambda)$.

In this article we prove similar results for the odd-dimensional orthogonal groups $O_{2n+1}(R)$ and the odd-dimensional unitary groups $U_{2n+1}(R,\Delta)$ under the assumption that $R$ is commutative and $n\geq 3$, cf. Theorem \ref{mthm1} and Theorem \ref{mthm2}. The proof of the orthogonal version is quite simple. The proof of the unitary version is a bit more complicated, but still it is much shorter than the proof of the SCT for the groups $U_{2n+1}(R,\Delta)$ given in \cite{bak-preusser} (on the other hand, in \cite{bak-preusser} the ring $R$ is only assumed to be quasi-finite and hence the result is a bit more general). For the odd-dimensional unitary groups $U_{2n+1}(R,\Delta)$ this yields the first proof of the SCT which does not use localization.

The rest of the paper is organized as follows. In Section 2 we recall some standard notation
which will be used throughout the paper. In Section 3 we state two lemmas which will be used in the proofs of the main theorems \ref{mthm1} and \ref{mthm2}. In Section 4 we recall the definitions of the odd-dimensional orthogonal group $O_{2n+1}(R)$ and some important subgroups, in Section 5 we prove Theorem \ref{mthm1}. In Section 6 we recall the definitions of the odd-dimensional unitary group $U_{2n+1}(R,\Delta)$ and some important subgroups, in Section 7 we prove Theorem \ref{mthm2}.

\section{Notation}
By a natural number we mean an element of the set $\mathbb{N}:=\{1,2,3,\dots\}$. If $G$ is a group and $g,h\in G$, we let $^hg:=hgh^{-1}$ and $[g,h]:=ghg^{-1}h^{-1}$. By a ring we will always mean an associative ring with $1$ such that $1\neq 0$. Ideal will mean two-sided ideal. If $X$ is a subset of a ring $R$, then we denote by $I(X)$ the ideal of $R$ generated by $X$. If $X=\{x\}$, then we may write $I(x)$ instead of $I(X)$. If $n$ is a natural number and $R$ is a ring, then the set of all $n\times n$ matrices with entries in $R$ is denoted by $M_n(R)$. If $a\in M_{n}(R)$, we denote the entry of $a$ at position $(i,j)$ by $a_{ij}$, the $i$-th row of $a$ by $a_{i*}$ and the $j$-th column of $a$ by $a_{*j}$. The group of all invertible matrices in $M_{n}(R)$ is denoted by $GL_n(R)$ and the identity element of $GL_n(R)$ by $e$. If $a\in GL_n(R)$, then the entry of $a^{-1}$ at position $(i,j)$ is denoted by $a'_{ij}$, the $i$-th row of $a^{-1}$ by $a'_{i*}$ and the $j$-th column of $a^{-1}$ by $a'_{*j}$. Further we denote by $^n\!R$ the set of all row vectors of length $n$ with entries in $R$ and by $R^n$ the set of all column vectors of length $n$ with entries in $R$. We consider $^n\!R$ as left $R$-module and $R^n$ as right $R$-module. If $u\in{}^n\!R$ (resp. $u\in R^n$), we denote by $u^t$ its transpose in $R^n$ (resp. in $^n\!R$).
\section{Preliminaries}\label{secpre}
The following two lemmas are easy to check.
\begin{lemma}\label{pre}
Let $G$ be a group and $a,b,c\in G$. Then ${}^{b^{-1}}[a,bc]=[b^{-1},a][a,c]$.
\end{lemma}
\begin{lemma}\label{pre2}
Let $G$ be a group, $E$ a subgroup and $a\in G$. Suppose that $b\in G$ is a product of $n$ elements of the form $^{\epsilon}a^{\pm 1}$ where $\epsilon\in E$. Then for any $\epsilon'\in E$
\begin{enumerate}[(i)]
\item $^{\epsilon'}b$ is a product of $n$ elements of the form $^{\epsilon}a^{\pm 1}$ and
\item $[\epsilon',b]$ is a product of $2n$ elements of the form $^{\epsilon}a^{\pm 1}$.
\end{enumerate}
\end{lemma}
Lemma \ref{pre2} will be used in the proofs of the main theorems without explicit reference.
\section{Odd-dimensional orthogonal groups}
In this section $n$ denotes a natural number and $R$ a commutative ring. First we recall the definitions of the odd-dimensional orthogonal group $O_{2n+1}(R)$ and the elementary subgroup $EO_{2n+1}(R)$. For an admissible pair $(I,J)$, we recall the definitions of the following subgroups of $O_{2n+1}(R)$; the preelementary subgroup $EO_{2n+1}(I,J)$ of level $(I,J)$, the elementary subgroup $EO_{2n+1}(R,I,J)$ of level $(I,J)$, the principal congruence subgroup $O_{2n+1}(R,I,J)$ of level $(I,J)$, and the full congruence subgroup $CO_{2n+1}(R,I,J)$ of level $(I,J)$.
\subsection{The odd-dimensional orthogonal group}
\begin{definition}\label{defog}
Set $M:=R^{2n+1}$. We use the following indexing for the elements of the standard basis of $M$: $(e_1,\dots,e_n,e_0,e_{-n},\dots,e_{-1})$.
That means that $e_i$ is the column whose $i$-th coordinate is one and all the other coordinates are zero if $1 \leq i\leq n$, the column whose $(n+1)$-th coordinate is one and all the other coordinates are zero if $i=0$, and the column whose $(2n+2+i)$-th coordinate is one and all the other coordinates are zero if $-n\leq i \leq -1$. We define the quadratic form
\begin{align*}
Q:M&\rightarrow R\\u&\mapsto u^t\begin{pmatrix} 0 & 0& p \\0& 1&0 \\0&0&0\end{pmatrix}u=u_1u_{-1}+\dots+u_nu_{-n}+u_0^2
\end{align*} 
where $p\in M_n(R)$ denotes the matrix with ones on the skew diagonal and zeros elsewhere.
The subgroup $O_{2n+1}(R):=\{\sigma\in GL_{2n+1}(R)\mid Q(\sigma u)=Q(u) ~\forall u\in M\}$ of $GL_{2n+1}(R)$ is called {\it (odd-dimensional) orthogonal group}.
\end{definition}
\begin{remark}
The odd-dimensional orthogonal groups $O_{2n+1}(R)$ are special cases of the odd-dimensional unitary groups $U_{2n+1}(R,\Delta)$, see \cite[Example 15]{bak-preusser}.
\end{remark}
\begin{definition}
Define $\Theta:=\{1,\dots,n,0,-n,\dots,-1\}$ and $\Theta_{hb}:=\Theta\setminus \{0\}$.
\end{definition}
\begin{lemma}\label{6}
Let $\sigma\in GL_{2n+1}(R)$. Then $\sigma\in O_{2n+1}(R)$ iff the conditions (i) and (ii) below hold. 
\begin{enumerate}[(i)]
\item \begin{align*}
       \sigma'_{ij}&=\sigma_{-j,-i}~\forall i,j\in\Theta_{hb},\\
       2 \sigma'_{0j}&=\sigma_{-j,0}~\forall j\in\Theta_{hb},\\
       \sigma'_{i0}&=2\sigma_{0,-i} ~\forall i\in\Theta_{hb} \text { and}\\
       2 \sigma'_{00}&=2\sigma_{00}.
      \end{align*}
\item \begin{align*}
Q(\sigma_{*j})=\delta_{0j}~\forall j\in \Theta.
\end{align*} 
\end{enumerate}
\end{lemma}
\begin{proof}
See \cite[Lemma 17]{bak-preusser}.
\end{proof}
\subsection{The polarity map}
\begin{definition}
The map
\begin{align*}
\widetilde{}: M&\longrightarrow M^t\\
u&\longmapsto \begin{pmatrix} u_{-1}&\dots&u_{-n}&2u_0&u_{n}&\dots&u_1\end{pmatrix}
\end{align*}
where $M^t={}^{2n+1}\!R$ is called {\it polarity map}. Clearly $~\widetilde{}~$ is linear, i.e. $\widetilde{u+v}=\tilde u+\tilde v$ and $\widetilde{ux}=x\tilde u$ for any $u,v\in M$ and $x\in R$.
\end{definition}
\begin{lemma}\label{8}
If $\sigma\in O_{2n+1}(R)$ and $u\in M$, then $\widetilde{\sigma u}=\tilde u\sigma^{-1}$.
\end{lemma}
\begin{proof}
Follows from Lemma \ref{6}.
\end{proof}
\subsection{The elementary subgroup}
If $i,j\in \Theta$, let $e^{ij}$ denote the matrix in $M_{2n+1}(R)$ with $1$ in the $(i,j)$-th position and $0$ in all other positions.
\begin{definition}
If $i,j\in \Theta_{hb}$, $i\neq\pm j$ and $x\in R$, the element  \[T_{ij}(x):=e+xe^{ij}-xe^{-j,-i}\] of $O_{2n+1}(R)$ is called an {\it (elementary) short root matrix}. 
If $i\in \Theta_{hb}$ and $x\in R$, the element \[T_{i}(x):=e+xe^{0,-i}-2xe^{i0}-x^2e^{i,-i}\] of $O_{2n+1}(R)$ is called an {\it (elementary) extra short root matrix}. If an element of $O_{2n+1}(R)$ is a short or extra short root matrix, then it is called {\it elementary matrix}. The subgroup of $O_{2n+1}(R)$ generated by all elementary matrices is called the {\it elementary subgroup} and is denoted by $EO_{2n+1}(R)$. 
\end{definition}
\begin{lemma}\label{10}
The following relations hold for elementary matrices.
\begin{align*}
&T_{ij}(x)=T_{-j,-i}(-x), \tag{S1}\\
&T_{ij}(x)T_{ij}(y)=T_{ij}(x+y), \tag{S2}\\
&[T_{ij}(x),T_{kl}(y)]=e \text{ if } k\neq j,-i \text{ and } l\neq i,-j, \tag{S3}\\
&[T_{ij}(x),T_{jk}(y)]=T_{ik}(xy) \text{ if } i\neq\pm k, \tag{S4}\\
&[T_{ij}(x),T_{j,-i}(y)]=e, \tag{S5}\\
&T_{i}(x)T_{i}(y)=T_{i}(x+y), \tag{E1}\\
&[T_{i}(x),T_{j}(y)]=T_{i,-j}(-2xy) \text{ if } i\neq\pm j, \tag{E2}\\
&[T_{i}(x),T_{i}(y)]=e, \tag{E3}\\
&[T_{ij}(x),T_{k}(y)]=e \text{ if } k\neq j,-i \text{ and} \tag{SE1}\\
&[T_{ij}(x),T_{j}(y)]=T_{j,-i}(-xy^2)T_{i}(xy).\tag{SE2}\\
\end{align*}
\end{lemma}
\begin{proof}
Straightforward computation.
\end{proof}
\begin{definition}
Let $u\in M$ be such that such that $u_{-1}=0$ and $u$ is {\it isotropic}, i.e. $Q(u)=0$. Then we denote the matrix
\begin{align*}
&\left(\begin{array}{cccc|c|cccc}1&-u_{-2}&\dots&-u_{-n}&-2u_0&-u_n&\dots&-u_2&0\\&1&&&&&&&u_2\\&&\ddots&&&&&&\vdots\\&&&1&&&&&u_n\\\hline &&&&1&&&&u_0\\\hline&&&&&1&&&u_{-n}\\&&&&&&\ddots&&\vdots\\&&&&&&&1&u_{-2}\\&&&&&&&&1\end{array}\right)\\
=&e+ue^t_{-1}- e_1\tilde u=T_{1}(u_0)T_{2,-1}(u_2)\dots T_{n,-1}(u_n)T_{-n,-1}(u_{-n})\dots T_{-2,-1}(u_{-2})\in EO_{2n+1}(R)
\end{align*}
by $T_{*,-1}(u)$. Clearly $T_{*,-1}(u)^{-1}=T_{*,-1}(-u)$ (note that $\tilde uu=0$ since $u$ is isotropic) and
\begin{equation}
^{\sigma}T_{*,-1}(u)=e+\sigma u\sigma'_{-1,*}- \sigma_{*1} \tilde u\sigma^{-1}=e+\sigma u\widetilde{\sigma_{*1}}- \sigma_{*1}\widetilde{\sigma u}\label{e1}
\end{equation} 
for any $\sigma\in O_{2n+1}(R)$, the last equality by Lemma \ref{8}.
\end{definition}
\begin{definition}\label{12}
Let $i,j\in\Theta_{hb}$ such that $i\neq\pm j$. Define
\begin{align*}
P_{ij}:=&e-e^{ii}-e^{jj}-e^{-i,-i}-e^{-j,-j}+e^{ij}-e^{ji}+e^{-i,-j}-e^{-j,-i}\\
=&T_{ij}(1)T_{ji}(-1)T_{ij}(1)\in EO_{2n+1}(R).                                                                                                                                                                                                                                                                                                                                                                                                                                                                                             
\end{align*}
It is easy to show that $(P_{ij})^{-1}=P_{ji}$. 
\end{definition}
\begin{lemma}\label{13}
Let $i,j,k\in\Theta_{hb}$ such that $i\neq \pm j$ and $k\neq \pm i,\pm j$. Let $x\in R$. Then 
\begin{enumerate}[(i)]
\item $^{P_{ki}}T_{ij}(x)=T_{kj}(x)$,
\item $^{P_{kj}}T_{ij}(x)=T_{ik}(x)$ and
\item $^{P_{-k,-i}}T_{i}(x)=T_{k}(x)$.
\end{enumerate}
\end{lemma}
\begin{proof}
Straightforward.
\end{proof}
\subsection{Relative elementary subgroups}
\begin{definition}
An admissible pair is a pair $(I,J)$ where $I$ and $J$ are ideals of $R$ such that 
\[2J,J^2\subseteq I\subseteq J\]
where $J^2=\{x^2\in R\mid x\in J\}$.
\end{definition}
\begin{definition}
Let $(I,J)$ denote an admissible pair. A short root matrix $T_{ij}(x)$ is called {\it $(I,J)$-elementary} if $x\in I$. An extra short root matrix $T_{i}(x)$ is called {\it $(I,J)$-elementary} if $x\in J$. If an element of $O_{2n+1}(R)$ is an $(I,J)$-elementary short or extra short root matrix, then it is called {\it $(I,J)$-elementary matrix}. The subgroup $EO_{2n+1}(I,J)$ of $EO_{2n+1}(R)$ generated by the $(I,J)$-elementary matrices is called the {\it preelementary subgroup of level $(I,J)$}. Its normal closure $EO_{2n+1}(R,I,J)$ in $EO_{2n+1}(R)$ is called the {\it elementary subgroup of level $(I,J)$}.
\end{definition}
\subsection{Congruence subgroups}
In this subsection $(I,J)$ denotes an admissible pair. 
\begin{definition}
The subgroup of $O_{2n+1}(R)$ consisting of all $\sigma\in O_{2n+1}(R)$ such that 
\begin{enumerate}[(i)]
\item $\sigma_{ij}\equiv \delta_{ij} \bmod I$ for any $i\in \Theta_{hb}, j\in\Theta$ and
\item $\sigma_{0j}\equiv \delta_{0j} \bmod J$ for any $j\in \Theta$
\end{enumerate}
is called {\it principal congruence subgroup of level} $(I,J)$ and is denoted by $O_{2n+1}(R,I,J)$. 
\end{definition}
\begin{theorem}
$O_{2n+1}(R,I,J)$ is a normal subgroup of $O_{2n+1}(R)$.
\end{theorem}
\begin{proof}
Follows from \cite[Corollary 35]{bak-preusser}.
\end{proof}
\begin{definition}
The subgroup
\[\{\sigma\in O_{2n+1}(R)\mid[\sigma,EO_{2n+1}(R)]\subseteq O_{2n+1}(R,I,J)\}\]
of $O_{2n+1}(R)$ is called {\it full congruence subgroup of level $(I,J)$} and is denoted by $CO_{2n+1}(R,I,J)$. 
\end{definition}
\begin{lemma}\label{19}
Let $\sigma\in O_{2n+1}(R)$. Then $\sigma\in CO_{2n+1}(R,I,J)$ iff
\begin{enumerate}[(i)]
\item $\sigma_{ij}\in I$ for any $i\in \Theta_{hb}, j\in\Theta$ such that $i\neq j$,
\item $\sigma_{0j}\in J$ for any $j\in \Theta_{hb}$,
\item $\sigma_{ii}-\sigma_{jj}\in I$ for any $i,j\in \Theta_{hb}$ and
\item $\sigma_{00}-\sigma_{jj}\in J$ for any $j \in\Theta_{hb}$.
\end{enumerate}
\end{lemma}
\begin{proof}
Straightforward.
\end{proof}
\begin{theorem}\label{23}
If $n\geq 3$, then the equalities
\begin{align*}[CO_{2n+1}(R,I,J),EO_{2n+1}(R)]=[EO_{2n+1}(R,I,J),EO_{2n+1}(R)]=EO_{2n+1}(R,I,J)\end{align*}
hold.
\end{theorem}
\begin{proof}
See \cite[Theorem 39]{bak-preusser}.
\end{proof}
\section{Sandwich classification for $O_{2n+1}(R)$}\label{sec6}
In this section $n$ denotes a natural number greater than or equal to $3$ and $R$ a commutative ring.
\begin{definition}
Let $\sigma\in O_{2n+1}(R)$. Then a matrix of the form $^{\epsilon}\sigma^{\pm 1}$ where $\epsilon\in EO_{2n+1}(R)$ is called an {\it elementary (orthogonal) $\sigma$-conjugate}.
\end{definition}
\begin{theorem}\label{mthm1}
Let $\sigma\in O_{2n+1}(R)$ and $i,j,k,l\in\Theta_{hb}$ such that $i\neq \pm j$ and $k\neq \pm l$. Then 
\begin{enumerate}[(i)]
\item $T_{kl}(\sigma_{ij})$ is a product of $8$ elementary orthogonal $\sigma$-conjugates,
\item $T_{kl}(\sigma_{i,-i})$ is a product of $16$ elementary orthogonal $\sigma$-conjugates,
\item $T_{kl}(\sigma_{i0})$ is a product of $24$ elementary orthogonal $\sigma$-conjugates,
\item $T_{kl}(2\sigma_{0j})$ is a product of $24$ elementary orthogonal $\sigma$-conjugates,
\item $T_{kl}(\sigma_{ii}-\sigma_{jj})$ is a product of $24$ elementary orthogonal $\sigma$-conjugates,
\item $T_{kl}(\sigma_{ii}-\sigma_{-i,-i})$ is a product of $48$ elementary orthogonal $\sigma$-conjugates,
\item $T_{k}(\sigma_{0j})$ is a product of $64n+148$ elementary orthogonal $\sigma$-conjugates and
\item $T_{k}(\sigma_{00}-\sigma_{jj})$ is a product of $192n+564$ elementary orthogonal $\sigma$-conjugates.
\end{enumerate}
\end{theorem}
\begin{proof}
(i) Set $\tau:=T_{21}(-\sigma_{23})T_{31}(\sigma_{22})T_{2,-3}(\sigma_{2,-1})$ and $\xi:={}^{\sigma}\tau^{-1}$. One checks easily that $(\sigma\tau^{-1})_{2*}=\sigma_{2*}$ and $(\tau^{-1}\sigma^{-1})_{*,-2}=\sigma'_{*,-2}$. Hence $\xi_{2*}=e^t_2$ and $\xi_{*,-2}=e_{-2}$. Set  
\begin{align*}
\zeta:={}^{\tau^{-1}}[T_{32}(1),[\tau,\sigma]]={}^{\tau^{-1}}[T_{32}(1),\tau\xi]=[\tau^{-1},T_{32}(1)][T_{32}(1),\xi],
\end{align*}
the last equality by Lemma \ref{pre}. One checks easily that $[\tau^{-1},T_{32}(1)]=T_{31}(-\sigma_{23})$ and $[T_{32}(1),\xi]=T_{-2}(x_{-2})\cdot$ $\cdot\prod\limits_{i\neq 0,\pm 2 }T_{i2}(x_i)$ for some $x_i\in R~(i\neq 0,2)$. Hence $\zeta=T_{31}(-\sigma_{23})T_{-2}(x_{-2})\prod\limits_{i\neq 0,\pm 2 }T_{i2}(x_i)$. It follows that $[T_{12}(1),\zeta]=T_{32}(\sigma_{23})$. Hence we have shown
\[[T_{12}(1),{}^{\tau^{-1}}[T_{32}(1),[\tau,\sigma]]]=T_{32}(\sigma_{23}).\]
This implies that $T_{32}(\sigma_{23})$ is a product of $8$ elementary $\sigma$-conjugates. It follows from Lemma \ref{13} that $T_{kl}(\sigma_{23})$ is a product of $8$ elementary $\sigma$-conjugates. Since one can bring $\sigma_{ij}$ to position $(2,3)$ conjugating by monomial matrices from $EO_{2n+1}(R)$ (see Definition \ref{12}), the assertion of (i) follows.\\
\\
(ii) Clearly the entry of $^{T_{ji}(1)}\sigma$ at position $(j,-i)$ equals $\sigma_{i,-i}+\sigma_{j,-i}$. Applying (i) to $^{T_{ji}(1)}\sigma$ we get that $T_{kl}(\sigma_{i,-i}+\sigma_{j,-i})$ is a product of $8$ elementary $\sigma$-conjugates (note that any elementary $^{T_{ji}(1)}\sigma$-conjugate is also an elementary $\sigma$-conjugate). Applying (i) to $\sigma$ we get that $T_{kl}(\sigma_{j,-i})$ is a product of $8$ elementary $\sigma$-conjugates. It follows that $T_{kl}(\sigma_{i,-i})=T_{kl}(\sigma_{i,-i}+\sigma_{j,-i})T_{ji}(-\sigma_{j,-i})$ is a product of $16$ elementary $\sigma$-conjugates.\\
\\
(iii) Clearly the entry of $^{T_{-j}(-1)}\sigma$ at position $(i,j)$ equals $\sigma_{i0}+\sigma_{ij}-\sigma_{i,-j}$. Applying (i) to $^{T_{-j}(-1)}\sigma$ we get that $T_{kl}(\sigma_{i0}+\sigma_{ij}-\sigma_{i,-j})$ is a product of $8$ elementary $\sigma$-conjugates. Applying (i) to $\sigma$ we get that $T_{kl}(-\sigma_{ij})$ and $T_{kl}(\sigma_{i,-j})$ each are a product of $8$ elementary $\sigma$-conjugates. It follows that $T_{kl}(\sigma_{i0})=T_{kl}(\sigma_{i0}+\sigma_{ij}-\sigma_{i,-j})T_{kl}(-\sigma_{ij})T_{kl}(\sigma_{i,-j})$ is a product of $24$ elementary $\sigma$-conjugates.\\
\\
(iv) Clearly the entry of $^{T_{i}(-1)}\sigma$ at position $(i,j)$ equals $2\sigma_{0j}+\sigma_{ij}-\sigma_{-i,j}$. Applying (i) to $^{T_{i}(-1)}\sigma$ we get that $T_{kl}(2\sigma_{0j}+\sigma_{ij}-\sigma_{-i,j})$ is a product of $8$ elementary $\sigma$-conjugates. Applying (i) to $\sigma$ we get that $T_{kl}(-\sigma_{ij})$ and $T_{kl}(\sigma_{-i,j})$ each are a product of $8$ elementary $\sigma$-conjugates. It follows that $T_{kl}(2\sigma_{0j})=T_{kl}(2\sigma_{0j}+\sigma_{ij}-\sigma_{-i,j})T_{kl}(-\sigma_{ij})T_{kl}(\sigma_{-i,j})$ is a product of $24$ elementary $\sigma$-conjugates.\\
\\
(v) Clearly the entry of $^{T_{ji}(1)}\sigma$ at position $(j,i)$ equals $\sigma_{ii}-\sigma_{jj}+\sigma_{ji}-\sigma_{ij}$. Applying (i) to $^{T_{ji}(1)}\sigma$ we get that $T_{kl}(\sigma_{ii}-\sigma_{jj}+\sigma_{ji}-\sigma_{ij})$ is a product of $8$ elementary $\sigma$-conjugates. Applying (i) to $\sigma$ we get that $T_{kl}(\sigma_{ij})$ and $T_{kl}(-\sigma_{ji})$ each are a product of $8$ elementary $\sigma$-conjugates. It follows that $T_{kl}(\sigma_{ii}-\sigma_{jj})=T_{kl}(\sigma_{ii}-\sigma_{jj}+\sigma_{ji}-\sigma_{ij})T_{kl}(\sigma_{ij})T_{kl}(-\sigma_{ji})$ is a product of $24$ elementary $\sigma$-conjugates.\\
\\
(vi) Follows from (v) since $T_{kl}(\sigma_{ii}-\sigma_{-i,-i})=T_{kl}(\sigma_{ii}-\sigma_{jj})T_{kl}(\sigma_{jj}-\sigma_{-i,-i})$.\\
\\
(vii) Set $m:=8$. In Step 1 we show that for any $x\in R$ the matrix $T_{k}(x\sigma_{0j}\sigma_{jj})$ is a product of $(2n+9)m+4$ elementary $\sigma$-conjugates. In Step 2 we use Step 1 in order to prove (v).\\
\\
{\bf Step 1.} Set $u':=\begin{pmatrix}0&\dots&0&\sigma'_{-1,-1}&-\sigma'_{-1,-2}\end{pmatrix}^t=\begin{pmatrix}0&\dots&0&\sigma_{11}&-\sigma_{21}\end{pmatrix}^t\in M$ and $u:=\sigma^{-1}u'\in M$. Then clearly $u_{-1}=0$. Further $Q(u)=Q(\sigma^{-1}u')=Q(u')=0$ and hence $u$ is isotropic. Set
\[
\xi:={}^{\sigma}T_{*,-1}(-u)\overset{(\ref{e1})}{=}e-\sigma u\widetilde{\sigma_{*1}}+ \sigma_{*1} \widetilde{\sigma u}=e-u'\widetilde{\sigma_{*1}}+ \sigma_{*1}\widetilde{u'}.
\]
Then
\begin{align*}
\xi=
\arraycolsep=8pt\def\arraystretch{1.5}\left(\begin{array}{cccccc|c|cccccc}
1-\sigma_{11}\sigma_{21}&\sigma_{11}\sigma_{11}&&&&&&&&&&&\\
-\sigma_{21}\sigma_{21}&1+\sigma_{21}\sigma_{11}&&&&&&&&&&&\\
-\sigma_{31}\sigma_{21}&\sigma_{31}\sigma_{11}&1&&&&&&&&&&\\
-\sigma_{41}\sigma_{21}&\sigma_{41}\sigma_{11}&&1&&&&&&&&&\\
\vdots&\vdots&&&\ddots&&&&&&&&\\
-\sigma_{n1}\sigma_{21}&\sigma_{n1}\sigma_{11}&&&&1&&&&&&&\\
\hline-\sigma_{01}\sigma_{21}&\sigma_{01}\sigma_{11}&&&&&1&&&&&&\\
\hline-\sigma_{-n,1}\sigma_{21}&\sigma_{-n,1}\sigma_{11}&&&&&&1&&&&&\\
\vdots&\vdots&&&&&&&\ddots&&&&\\
-\sigma_{-4,1}\sigma_{21}&\sigma_{-4,1}\sigma_{11}&&&&&&&&1&&&\\
-\sigma_{-3,1}\sigma_{21}&\sigma_{-3,1}\sigma_{11}&&&&&&&&&1&&\\
-\alpha&0&*&*&\dots&*&*&*&\dots&*&-\sigma_{31}\sigma_{11}&*&*\\
0&\alpha&*&*&\dots&*&*&*&\dots&*&\sigma_{31}\sigma_{21}&*&*
\end{array}\right)
\end{align*}
where $\alpha=\sigma_{-1,1}\sigma_{11}+\sigma_{-2,1}\sigma_{21}$. Set \[\tau:=T_{-3,1}(\sigma_{-3,1}\sigma_{21})T_{-3,2}(-\sigma_{-3,1}\sigma_{11}).\] 
It follows from (i) that $\tau$ is a product of $2m$ elementary $\sigma$-conjugates. Clearly
\begin{align*}
\xi\tau=
\arraycolsep=8pt\def\arraystretch{1.5}\left(\begin{array}{cccccc|c|cccccc}
1-\sigma_{11}\sigma_{21}&\sigma_{11}\sigma_{11}&&&&&&&&&&&\\
-\sigma_{21}\sigma_{21}&1+\sigma_{21}\sigma_{11}&&&&&&&&&&&\\
-\sigma_{31}\sigma_{21}&\sigma_{31}\sigma_{11}&1&&&&&&&&&&\\
-\sigma_{41}\sigma_{21}&\sigma_{41}\sigma_{11}&&1&&&&&&&&&\\
\vdots&\vdots&&&\ddots&&&&&&&&\\
-\sigma_{n1}\sigma_{21}&\sigma_{n1}\sigma_{11}&&&&1&&&&&&&\\
\hline-\sigma_{01}\sigma_{21}&\sigma_{01}\sigma_{11}&&&&&1&&&&&&\\
\hline-\sigma_{-n,1}\sigma_{21}&\sigma_{-n,1}\sigma_{11}&&&&&&1&&&&&\\
\vdots&\vdots&&&&&&&\ddots&&&&\\
-\sigma_{-4,1}\sigma_{21}&\sigma_{-4,1}\sigma_{11}&&&&&&&&1&&&\\
0&0&&&&&&&&&1&&\\
-\alpha-\beta&\delta&0&*&\dots&*&*&*&\dots&*&*&*&*\\
\gamma&\alpha-\beta&0&*&\dots&*&*&*&\dots&*&*&*&*
\end{array}\right)
\end{align*}
where $\beta=\sigma_{-3,1}\sigma_{21}\sigma_{31}\sigma_{11}$, $\gamma=\sigma_{-3,1}\sigma_{21}\sigma_{31}\sigma_{21}$ and $\delta=\sigma_{-3,1}\sigma_{11}\sigma_{31}\sigma_{11}$. Let $x\in R$ and set  
\begin{align*}
\zeta:={}&^{T_{*,-1}(-u)}[T_{2,-3}(-x),[T_{*,-1}(u),\sigma]\tau]\\
={}&^{T_{*,-1}(-u)}[T_{2,-3}(-x),T_{*,-1}(u)\xi\tau]
\\=&[T_{*,-1}(-u),T_{2,-3}(-x)][T_{2,-3}(-x),\xi\tau],
\end{align*}
the last equality by Lemma \ref{pre}.
Clearly $\zeta$ is a product of $4m+4$ elementary $\sigma$-conjugates. One checks easily that 
\begin{align*}
&[T_{*,-1}(-u),T_{2,-3}(-x)]\\
=&T_{1,-2}(xu_{-3})T_{1,-3}(-xu_{-2})\\
=&T_{1,-2}(x(\sigma_{23}\sigma_{11}-\sigma_{13}\sigma_{21}))T_{1,-3}(-x(\sigma_{22}\sigma_{11}-\sigma_{12}\sigma_{21})).
\end{align*}
Further
\[[T_{2,-3}(-x),\xi\tau]=(\prod\limits_{\substack{p=1,\\p\neq 3,0}}^{-4}T_{p,-3}(x\sigma_{p1}\sigma_{11}))T_{-2,-3}(x\delta)T_{-1,-3}(x(\alpha-\beta))T_{3}(x\sigma_{01}\sigma_{11}).\]
Hence
\begin{align*}
\zeta=&T_{1,-2}(x(\sigma_{23}\sigma_{11}-\sigma_{13}\sigma_{21}))T_{1,-3}(-x((\sigma_{22}-\sigma_{11})\sigma_{11}-\sigma_{12}\sigma_{21}))\cdot\\
&\cdot (\prod\limits_{\substack{p=2,\\p\neq 3,0}}^{-4}T_{p,-3}(x\sigma_{p1}\sigma_{11}))T_{-2,-3}(x\delta)T_{-1,-3}(x(\alpha-\beta))T_{3}(x\sigma_{01}\sigma_{11}).
\end{align*}
It follows from (i), (ii) and (v) that $T_{3}(x\sigma_{01}\sigma_{11})$ is a product of $4m+4+2m+4m+(2n-5)m+m+3m=(2n+9)m+4$ elementary $\sigma$-conjugates. In view of Definition \ref{12} and Lemma \ref{13} we get that $T_{k}(x\sigma_{0j}\sigma_{jj})$ is a product of $(2n+9)m+4$ elementary $\sigma$-conjugates.\\
\\
{\bf Step 2.} Clearly
\begin{align*}
&T_{k}(\sigma_{0j})=T_{k}((\sum\limits_{s\in\Theta}\sigma'_{js}\sigma_{sj})\sigma_{0j})=\prod\limits_{s\in\Theta}T_{k}(\sigma'_{js}\sigma_{sj}\sigma_{0j}).
\end{align*}
By (i) and relation (SE2) in Lemma \ref{10}, $T_{k}(\sigma'_{js}\sigma_{sj}\sigma_{0j})$ is a product of $3m$ elementary $\sigma$-conjugates if $s\neq \pm j,0$. By (ii) and relation (SE2), $T_{k}(\sigma'_{j,-j}\sigma_{-j,j}\sigma_{0j})$ is a product of $3\cdot 2m=6m$ elementary $\sigma$-conjugates. By (iv) and relation (SE2), $T_{k}(\sigma'_{j0}\sigma_{0j}\sigma_{0j})$ is a product of $3\cdot 3m=9m$ elementary $\sigma$-conjugates (note that $\sigma'_{j0}$ is a multiple of $2$ by Lemma \ref{6}). By Step 1, $T_{k}(\sigma'_{jj}\sigma_{jj}\sigma_{0j})$ is a product of $(2n+9)m+4$ elementary $\sigma$-conjugates.  Thus $T_{k}(\sigma_{0j})$ is a product of $(2n-2)\cdot 3m+6m+9m+(2n+9)m+4=8nm+18m+4=64n+148$ elementary $\sigma$-conjugates.\\
\\
(viii) One checks easily that the entry of $^{T_{-j}(1)}\sigma$ at position $(0,j)$ equals $\sigma_{jj}-\sigma_{00}+\sigma_{0j}-\sigma_{j0}-\sigma_{0,-j}-\sigma_{j,-j}$. By applying (vii) to $^{T_{-j}(1)}\sigma$ we get that $T_{k}(\sigma_{jj}-\sigma_{00}+\sigma_{0j}-\sigma_{j0}-\sigma_{0,-j}-\sigma_{j,-j})$ is a product of $64n+148$ elementary $\sigma$-conjugates. By applying (vii) to $\sigma$ we get that $T_{k}(-\sigma_{0j})$ and $T_{k}(\sigma_{0,-j})$ each are a product of $64n+148$ elementary $\sigma$-conjugates. By (iii) and relation (SE2) in Lemma \ref{10}, $T_{k}(\sigma_{j0})$ is a product of $3\cdot 24=72$ elementary $\sigma$-conjugates. By (ii) and relation (SE2) in Lemma \ref{10}, $T_{k}(\sigma_{j,-j})$ is a product of $3\cdot 16=48$ elementary $\sigma$-conjugates. It follows that $T_{k}(\sigma_{jj}-\sigma_{00})=T_{k}(\sigma_{jj}-\sigma_{00}+\sigma_{0j}-\sigma_{j0}-\sigma_{0,-j}-\sigma_{j,-j})T_{k}(-\sigma_{0j})T_{k}(\sigma_{j0})T_{k}(\sigma_{0,-j})T_{k}(\sigma_{j,-j})$ is a product of $3\cdot(64n+148)+72+48=192n+564$ elementary $\sigma$-conjugates.
\end{proof}

As a corollary we get the Sandwich Classification Theorem for $O_{2n+1}(R)$.
\begin{corollary}
Let $H$ be a subgroup of $O_{2n+1}(R)$. Then $H$ is normalized by $EO_{2n+1}(R)$ iff 
\begin{equation}
EO_{2n+1}(R,I,J)\subseteq H\subseteq CO_{2n+1}(R,I,J)
\end{equation}
for some admissible pair $(I,J)$.
\end{corollary}
\begin{proof}
First suppose that $H$ is normalized by $EO_{2n+1}(R)$. Let $(I,J)$ be the admissible pair defined by $I:=\{x\in R\mid T_{12}(x)\in H\}$ and $J:=\{x\in R\mid T_{1}(x)\in H\}$. Then clearly $EO_{2n+1}(R,I,J)\subseteq H$. It remains to show that $H\subseteq CO_{2n+1}(R,I,J)$, i.e. that if $\sigma\in H$,  then the conditions (i)-(iv) in Lemma \ref{19} are satisfied. But that follows from the previous theorem. Suppose now that (2) holds for some admissible pair $(I,J)$. Then it follows from the standard commutator formula in Theorem \ref{23} that $H$ is normalized by $EO_{2n+1}(R)$.
\end{proof}
\section{Odd-dimensional unitary groups}
We describe Hermitian form rings $(R,\Delta)$ and odd form ideals $(I,\Omega)$ first, then the odd-dimensional unitary group $U_{2n+1}(R,\Delta)$ and its elementary subgroup $EU_{2n+1}(R,\Delta)$ over a Hermitian form ring $(R,\Delta)$. For an odd form ideal $(I,\Omega)$, we recall the definitions of the following subgroups of $U_{2n+1}(R,\Delta)$; the preelementary subgroup $EU_{2n+1}(I, \Omega)$ of level $(I,\Omega)$, the elementary subgroup $EU_{2n+1}((R,\Lambda),(I,\Omega))$ of level $(I,\Omega)$, the principal congruence subgroup $U_{2n+1}((R,\Lambda),(I,\Omega))$ of level $(I,\Omega)$, the normalized principal congruence subgroup $NU_{2n+1}((R,\Lambda),(I,\Omega))$ of level $(I,\Omega)$, and the full congruence subgroup $CU_{2n+1}((R,\Lambda),(I,\Omega))$ of level $(I,\Omega)$.
\subsection{Hermitian form rings and odd form ideals}\label{sec 3.1}
First we recall the definitions of a ring with involution with symmetry and a Hermitian ring.
\begin{definition}
Let $R$ be a ring and 
\begin{align*}
\bar{}:R&\rightarrow R\\
x&\mapsto \bar{x}
\end{align*}
an anti-isomorphism of $R$ (i.e. $\bar{}~$ is bijective, $\overline{x+y}=\bar x+\bar y$, $\overline{xy}=\bar y\bar x$ for any $x,y\in R$ and $\bar 1=1$). Further let $\lambda\in R$ such that $\bar{\bar x}=\lambda x\bar\lambda$ for any $x\in R$. Then $\lambda$ is called a {\it symmetry} for $~\bar{}~$, the pair $(~\bar{}~,\lambda)$ an {\it involution with symmetry} and the triple $(R,~\bar{}~,\lambda)$ a {\it ring with involution with symmetry}. A subset $A\subseteq R$ is called {\it involution invariant} iff $\bar x\in A$ for any $x\in A$. We call a quadruple $(R,~\bar{}~,\lambda,\mu )$ where $(R,~\bar{}~,\lambda)$ is a ring with involution with symmetry and $\mu \in R$ such that $\mu =\bar\mu \lambda$ a {\it Hermitian ring}.
\end{definition}
\begin{remark}\label{25}
Let $(R,~\bar{}~,\lambda,\mu )$ be a Hermitian ring.
\begin{enumerate}[(a)]
\item It is easy to show that $\bar\lambda=\lambda^{-1}$.
\item The map
\begin{align*}
\b{}:R&\rightarrow R\\
x&\mapsto \b{x}:=\bar\lambda \bar x\lambda
\end{align*}
is the inverse map of $~\bar{}~$. One checks easily that $(R,~\b{}~,\b{$\lambda$},\b{$\mu $})$ is a Hermitian ring.
\end{enumerate}
\end{remark}

Next we recall the definition of an $R^{\bullet}$-module.
\begin{definition}
If $R$ is a ring, let $R^\bullet$ denote the underlying set of the ring equipped with the  
multiplication of the ring, but not the addition of the ring. A {\it (right) $R^{\bullet}$-module} is a not 
necessarily abelian group $(G,\+)$ equipped with a map 
\begin{align*}
\circ: G\times R^{\bullet}&\rightarrow G\\
(a,x) &\mapsto a\circ x
\end{align*}
such that the following holds:
\begin{enumerate}[(i)]
\item $a\circ 0=0$ for any $a\in G$,
\item $a\circ 1=a$ for any $a\in G$,
\item $(a\circ x)\circ y=a\circ (xy)$ for any $a\in G$ and $x,y\in R$ and
\item $(a\+ b)\circ x=(a\circ x)\+(b\circ x)$ for any $a,b\in G$ and $x\in R$.
\end{enumerate}
A left $R^{\bullet}$-module is defined analogously. An $R$-module is canonically an $R^{\bullet}$-module, but not conversely. Let $G$ and $G'$ be $R^{\bullet}$-modules. A group homomorphism $f:G\rightarrow G'$ satisfying $f(a\circ x)=f(a)\circ x$ for any $a\in G$ and $x\in R$ is called a {\it  homomorphism of $R^{\bullet}$-modules}. A subgroup $H$ of $G$ which is $\circ$-stable (i.e. $a\circ x\in H$ for any $a\in H$ and $x\in R$) is called an {\it $R^{\bullet}$-submodule}. Further, if $A\subseteq G$ and $B\subseteq R$, we denote by $A\circ B$ the subgroup of $G$ generated by $\{a\circ b\mid a\in A,b\in B\}$. We treat $\circ$ as an operator with higher priority than $\+$.
\end{definition}

An important example of an $R^{\bullet}$-module is the Heisenberg group, which we define next. The odd form parameters $\Delta$ which are used to define the odd-dimensional unitary groups are certain $R^{\bullet}$-submodules of the Heisenberg group. 
\begin{definition}\label{27}
Let $(R,~\bar{}~,\lambda,\mu )$ be a Hermitian ring. Define the map.
\begin{align*}
\+: (R\times R)\times (R\times R) &\rightarrow R\times R\\
((x_1,y_1),(x_2,y_2))&\mapsto (x_1,y_1)\+ (x_2,y_2):=(x_1+x_2,y_1+y_2-\bar x_1\mu  x_2).
\end{align*}
Then $(R\times R,\+)$ is a group, which we call the {\it Heisenberg group} and denote by $\h$. Equipped with the map
\begin{align*}
\circ:(R\times R)\times R^{\bullet}&\rightarrow R\times R\\
((x,y),a)&\mapsto (x,y)\circ a:=(xa,\bar aya)
\end{align*}
$\h$ becomes an $R^{\bullet}$-module. 
\end{definition}
\begin{remark}
We denote the inverse of an element $(x,y)\in \h$ by $\minus(x,y)$. One checks easily that $\minus(x,y)=(-x,-y-\bar x\mu  x)$ for any $(x,y)\in \h$.
\end{remark}

In order to define the odd-dimensional unitary groups we need the notion of a Hermitian form ring.
\begin{definition}
Let $(R,~\bar{}~,\lambda,\mu )$ be a Hermitian ring. Let $(R,+)$ have the $R^{\bullet}$-module structure defined by $x\circ a = \bar{a}xa$. Define the {\it trace map}
\begin{align*}
tr:\h&\rightarrow R\\
(x,y)&\mapsto \bar x\mu  x+y+\bar y\lambda.
\end{align*}
One checks easily that $tr$ is a homomorphism of $R^{\bullet}$-modules. Set \[\Delta_{min}:=\{(0,x-\overline{x}\lambda)\mid x\in R\}\] and \[\Delta_{max}:=ker(tr).\] An $R^{\bullet}$-submodule $\Delta$ of $\h$ lying between $\Delta_{min}$ and $\Delta_{max}$ is called an {\it odd form parameter of $(R,~\bar{}~,\lambda,\mu )$}. Since $\Delta_{min}$ and $\Delta_{max}$ are $R^{\bullet}$-submodules of $\h$, they are respectively the smallest and largest odd form parameters. A pair $((R,~\bar{}~,\lambda,\mu ),\Delta)$ is called a {\it Hermitian form ring}. We shall usually abbreviate it by $(R,\Delta)$. 
\end{definition}
 
Next we define an odd form ideal of a Hermitian form ring. 
\begin{definition}
Let $(R,\Delta)$ be a Hermitian form ring and $I$ an involution invariant ideal of $R$. Set $J(\Delta):=\{y\in R\mid\exists z\in R:(y,z)\in \Delta\}$ and $\tilde I:=\{x\in R\mid\overline{J(\Delta)}\mu  x\subseteq I\}$. Obviously $\tilde I$ and $J(\Delta)$ are right ideals of $R$ and $I\subseteq \tilde I$. Further set \[\Omega^I_{min}:=\{(0,x-\bar x\lambda)\mid x\in I\}\+ \Delta\circ I\] and \[\Omega^I_{max}:=\Delta\cap (\tilde I\times  I).\]
An $R^{\bullet}$-submodule $\Omega$ of $\h$ lying between $\Omega^I_{min}$ and $\Omega^I_{max}$ is called a {\it relative odd form parameter of level $I$}. Since $\Omega^I_{min}$ and $\Omega^I_{max}$ are $R^{\bullet}$-submodules of $\h$, they are respectively the smallest and the largest relative odd form parameters of level $I$. If $\Omega$ is a relative odd form parameter of level $I$, then $(I,\Omega)$ is called an {\it odd form ideal of $(R,\Delta)$}.
\end{definition}

The following lemma is straightforward to check. It will be used in the proof of Theorem \ref{mthm2}. 
\begin{lemma}\label{last}
Let $(R,~\bar{}~,\lambda,\mu )$ be a Hermitian ring, $(a,b)\in \Delta_{max}$, $n\in\mathbb{N}$ and $x_1,\dots,x_n\in R$. Then
\[(a,b)\circ \sum\limits_{i=1}^nx_i=(\plus\limits_{i=1}^n (a,b)\circ x_i)\+(0,\sum\limits_{\substack{i,j=1,\\i>j}}^n x_ib\bar x_j-\overline{x_ib\bar x_j}\lambda).\]
\end{lemma}
\subsection{The odd-dimensional unitary group}
Let $(R,\Delta)$ be a Hermitian form ring and $n\in \mathbb{N}$. Let $M$, $e_1,\dots,e_n,e_0,e_{-n},\dots,e_{-1}$ and $p$ be defined as in Definition \ref{defog}. If $u\in M$, then we call $(u_1,\dots,u_n,u_{-n},\dots,$ $u_{-1})^t\in R^{2n}$ the {\it hyperbolic part of $u$} and denote it by $u_{hb}$. Further we set $u^*:=\bar u^t$ and $u_{hb}^*:=\bar u_{hb}^t$. Define the maps
\begin{align*}
b:M\times M&\rightarrow R\\
(u,v)&\mapsto u^*\begin{pmatrix} 0& 0 & p\\0&\mu &0\\ p\lambda &0 &0 \end{pmatrix}v=\sum\limits_{i=1}^{n}\bar u_i v_{-i}+\bar u_0\mu  v_0+\sum\limits_{i=-n}^{-1}\bar u_{i}\lambda v_{-i}
\end{align*}
and 
\begin{align*}
q:M&\rightarrow \h\\
u&\mapsto (q_1(u),q_2(u)):=(u_0,u_{hb}^*\begin{pmatrix} 0&p\\0&0 \end{pmatrix}u_{hb})=(u_0,\sum\limits_{i=1}^{n}\bar u_i u_{-i}).
\end{align*}

\begin{lemma}
~\\
\vspace{-0.6cm}
\begin{enumerate}[(i)]
\item $b$ is a \textnormal{$\lambda$-Hermitian form}, i.e. $b$ is biadditive, $b(ux,vy)=\bar x b(u,v) y~\forall u,v\in M,x,y\in R$ and $b(u,v)=\overline{b(v,u)}\lambda~\forall u,v\in M$.
\item $q(ux)=q(u)\circ x~\forall u\in M, x\in R$, $q(u+v)\equiv q(u)\+ q(v)\+(0,b(u,v))\bmod \Delta_{min}~\forall u,v\in M$ and $tr(q(u))=b(u,u)~\forall u\in M$.
\end{enumerate}
\end{lemma}
\begin{proof}
{Straightforward computation.}
\end{proof}
\begin{definition}
The group
\[U_{2n+1}(R,\Delta):=\{\sigma\in GL_{2n+1}(R)\mid b(\sigma u,\sigma v)=b(u,v)~\forall u,v\in M \text{ and }q(\sigma u)\equiv q(u)\bmod \Delta~\forall u\in M\}\]
is called the {\it odd-dimensional unitary group}.
\end{definition}
\begin{remark}
The odd-dimensional unitary groups $U_{2n+1}(R,\Delta)$ are isomorphic to Petrov's odd hyperbolic unitary groups with $V_0=R$ (see \cite{petrov}). Namely $U_{2n+1}(R,\Delta)$ is isomorphic to Petrov's group $U_{2l}(R,\mathfrak{L})$ where Petrov's pseudoinvolution $~\hat{}~$ is defined by $\hat x=-\bar x\lambda$, $V_0=R$, $B_0(a,b)=\hat a\hat 1^{-1}\mu  b$, $\mathfrak{L}=\{(x,y)\in R\times R\mid(x,-y)\in\Delta\}$ and $l=n$.
\end{remark}
\begin{example}\label{34}
The groups $U_{2n+1}(R,\Delta)$ include as special cases the even-dimensional unitary groups $U_{2n}(R,\Lambda)$ where $n\in\mathbb{N}$ and $(R,\Lambda)$ is a form ring, the general linear groups $GL_{n}(R)$ where $R$ is any ring and $n\geq 2$ and further the symplectic groups $Sp_{n}(R)$ and the orthogonal groups $O_{n}(R)$ where $R$ is a commutative ring and $n\geq 2$. For details see \cite[Example 15]{bak-preusser}. 
\end{example}
\begin{definition}
Define $\Theta_+:=\{1,\dots,n\}$, $\Theta_-:=\{-n,\dots,-1\}$, $\Theta:=\Theta_+\cup\Theta_-\cup\{0\}$, $\Theta_{hb}:=\Theta\setminus \{0\}$ and the map 
\begin{align*}\epsilon:\Theta_{hb} &\rightarrow\{\pm 1\}\\
i&\mapsto\begin{cases} 1, &\mbox{if } i\in\Theta_+, \\ 
-1, & \mbox{if } i\in\Theta_-. \end{cases}
\end{align*}
Further if $i,j\in \Theta$, we write $i<j$ iff either $i\in \Theta_+\land j=0$ or $i=0\land j\in \Theta_-$ or $i,j\in \Theta_+\land i<j$ or $i,j\in \Theta_-\land i<j$ or $i\in \Theta_+\land j\in \Theta_-$.
\end{definition}
\begin{lemma}\label{36}
Let $\sigma\in GL_{2n+1}(R)$. Then $\sigma\in U_{2n+1}(R,\Delta)$ iff the conditions (i) and (ii) below hold. 
\begin{enumerate}[(i)]
\item \begin{align*}
       \sigma'_{ij}&=\lambda^{-(\epsilon(i)+1)/2}\bar\sigma_{-j,-i}\lambda^{(\epsilon(j)+1)/2}~\forall i,j\in\Theta_{hb},\\
       \mu \sigma'_{0j}&=\bar\sigma_{-j,0}\lambda^{(\epsilon(j)+1)/2}~\forall j\in\Theta_{hb},\\
       \sigma'_{i0}&=\lambda^{-(\epsilon(i)+1)/2}\bar\sigma_{0,-i}\mu ~\forall i\in\Theta_{hb} \text { and}\\
       \mu \sigma'_{00}&=\bar\sigma_{00}\mu .
      \end{align*}
\item \begin{align*}
q(\sigma_{*j})\equiv (\delta_{0j},0)\bmod \Delta ~\forall j\in \Theta.
\end{align*} 
\end{enumerate}
\end{lemma}
\begin{proof}
See \cite[Lemma 17]{bak-preusser}.
\end{proof}
\subsection{The polarity map}
\begin{definition}
The map
\begin{align*}
\widetilde{}: M&\longrightarrow M^*\\
u&\longmapsto \begin{pmatrix} \bar u_{-1}\lambda&\dots&\bar u_{-n}\lambda&\bar u_0\mu&\bar u_{n}&\dots&\bar u_1\end{pmatrix}
\end{align*}
where $M^*={}^{2n+1}\!R$ is called {\it polarity map}. Clearly $~\widetilde{}~$ is {\it involutary linear}, i.e. $\widetilde{u+v}=\tilde u+\tilde v$ and $\widetilde{ux}=\bar x\tilde u$ for any $u,v\in M$ and $x\in R$.
\end{definition}
\begin{lemma}\label{38}
If $\sigma\in U_{2n+1}(R,\Delta)$ and $u\in M$, then $\widetilde{\sigma u}=\tilde u\sigma^{-1}$.
\end{lemma}
\begin{proof}
Follows from Lemma \ref{36}.
\end{proof}
\subsection{The elementary subgroup}
We introduce the following notation. In Definition \ref{27} we defined an $R^{\bullet}$-module structure on $\h$. Let $(R,~\b{}~,\b{$\lambda$},\b{$\mu $})$ be the Hermitian ring defined in Remark \ref{25}(b). Let $\underline{\h}$ denote the Heisenberg group corresponding to $(R,~\b{}~,\b{$\lambda$},\b{$\mu $})$. The underlying set of both $\h$ and $\underline{\h}$ is $R\times R$. If we replace the Hermitian ring $(R,~\bar~,\lambda,\mu)$ by the Hermitian ring $(R,~\b{}~,\b{$\lambda$},\b{$\mu $})$, then we get another $R^{\bullet}$-module structure on $R\times R$. We denote the group operation (resp. scalar multiplication) defined by $(R,~\bar{}~,\lambda,\mu )$ on $R\times R$ by $\+_1$ (resp. $\circ_1$) and the group operation (resp. scalar multiplication) defined by $(R,~\b{}~,\b{$\lambda$},\b{$\mu $})$ on $R\times R$ by $\+_{-1}$ (resp. $\circ_{-1}$). Further we set $\Delta^1:=\Delta$ and $\Delta^{-1}:=\{(x,y)\in R\times R\mid(x,\bar y)\in \Delta\}$. One checks easily that $((R,~\b{}~,\b{$\lambda$},\b{$\mu $}),\Delta^{-1})$ is a Hermitian form ring. Analogously, if $(I,\Omega)$ is an odd form ideal of $(R,\Delta)$, we set $\Omega^1:=\Omega$ and $\Omega^{-1}:=\{(x,y)\in R\times R\mid(x,\bar y)\in \Omega\}$. One checks easily that $(I,\Omega^{-1})$ is an odd form ideal of $(R,\Delta^{-1})$. 
 
If $i,j\in \Theta$, let $e^{ij}$ denote the matrix in $M_{2n+1}(R)$ with $1$ in the $(i,j)$-th position and $0$ in all other positions.
\begin{definition}
If $i,j\in \Theta_{hb}$, $i\neq\pm j$ and $x\in R$, the element  \[T_{ij}(x):=e+xe^{ij}-\lambda^{(\epsilon(j)-1)/2}\bar x\lambda^{(1-\epsilon(i))/2}e^{-j,-i}\] of $U_{2n+1}(R,\Delta)$ is called an {\it (elementary) short root matrix}. 
If $i\in \Theta_{hb}$ and $(x,y)\in \Delta^{-\epsilon(i)}$, the element \[T_{i}(x,y):=e+xe^{0,-i}-\lambda^{-(1+\epsilon(i))/2}\bar x\mu e^{i0}+ye^{i,-i}\] of $U_{2n+1}(R,\Delta)$ is called an {\it (elementary) extra short root matrix}. The extra short root matrices of the kind \[T_{i}(0,y)=e+ye^{i,-i}\] are called {\it (elementary) long root matrices}. If an element of $U_{2n+1}(R,\Delta)$ is a short or extra short root matrix, then it is called {\it elementary matrix}. The subgroup of $U_{2n+1}(R,\Delta)$ generated by all elementary matrices is called the {\it elementary subgroup} and is denoted by $EU_{2n+1}(R,\Delta)$. 
\end{definition}
\begin{lemma}\label{39}
The following relations hold for elementary matrices.
\begin{align*}
&T_{ij}(x)=T_{-j,-i}(-\lambda^{(\epsilon(j)-1)/2}\bar x\lambda^{(1-\epsilon(i))/2}), \tag{S1}\\
&T_{ij}(x)T_{ij}(y)=T_{ij}(x+y), \tag{S2}\\
&[T_{ij}(x),T_{kl}(y)]=e \text{ if } k\neq j,-i \text{ and } l\neq i,-j, \tag{S3}\\
&[T_{ij}(x),T_{jk}(y)]=T_{ik}(xy) \text{ if } i\neq\pm k, \tag{S4}\\
&[T_{ij}(x),T_{j,-i}(y)]=T_{i}(0,xy-\lambda^{(-1-\epsilon(i))/2}\bar y\bar x\lambda^{(1-\epsilon(i))/2}), \tag{S5}\\
&T_{i}(x_1,y_1)T_{i}(x_2,y_2)=T_{i}((x_1,y_1)\+_{-\epsilon(i)}(x_2,y_2)), \tag{E1}\\
&[T_{i}(x_1,y_1),T_{j}(x_2,y_2)]=T_{i,-j}(-\lambda^{-(1+\epsilon(i))/2}\bar x_1\mu x_2) \text{ if } i\neq\pm j, \tag{E2}\\
&[T_{i}(x_1,y_1),T_{i}(x_2,y_2)]=T_{i}(0,-\lambda^{-(1+\epsilon(i))/2}(\bar x_1\mu x_2-\bar x_2\mu x_1)), \tag{E3}\\
&[T_{ij}(x),T_{k}(y,z)]=e \text{ if } k\neq j,-i \text{ and} \tag{SE1}\\
&[T_{ij}(x),T_{j}(y,z)]=T_{j,-i}(z\lambda^{(\epsilon(j)-1)/2}\bar x\lambda^{(1-\epsilon(i))/2})T_{i}(y\lambda^{(\epsilon(j)-1)/2}\bar x\lambda^{(1-\epsilon(i))/2},xz\lambda^{(\epsilon(j)-1)/2}\bar x\lambda^{(1-\epsilon(i))/2}).\tag{SE2}\\
\end{align*}
\end{lemma}
\begin{proof}
Straightforward computation.
\end{proof}
\begin{definition}
Let $u\in M$ be such that such that $u_{-1}=0$ and $u$ is {\it isotropic}, i.e. $q(u)\in\Delta$. Then we denote the matrix
\begin{align*}
&\left(\begin{array}{cccc|c|cccc}1&-\bar\lambda\bar u_{-2}\lambda&\dots&-\bar\lambda\bar u_{-n}\lambda&-\bar\lambda\bar u_0\mu&-\bar\lambda\bar u_n&\dots&-\bar\lambda\bar u_2&u_1-\bar\lambda\bar u_1\\&1&&&&&&&u_2\\&&\ddots&&&&&&\vdots\\&&&1&&&&&u_n\\\hline &&&&1&&&&u_0\\\hline&&&&&1&&&u_{-n}\\&&&&&&\ddots&&\vdots\\&&&&&&&1&u_{-2}\\&&&&&&&&1\end{array}\right)\\
=&e+ue^t_{-1}- e_1\bar\lambda\tilde u=T_{1}(q_1(u),\bar\lambda(q_2(u)-\bar u_1+\lambda u_1))\prod\limits_{\substack{i=2,\\i\neq 0}}^{-2}T_{i,-1}(u_i)\in EU_{2n+1}(R,\Delta)
\end{align*}
by $T_{*,-1}(u)$. Clearly $T_{*,-1}(u)^{-1}=T_{*,-1}(-u)$ (note that $\tilde uu=tr(q(u))=0$ since $u$ is isotropic) and
\begin{equation}
^{\sigma}T_{*,-1}(u)=e+\sigma u\sigma'_{-1,*}- \sigma_{*1} \bar\lambda\tilde u\sigma^{-1}=e+\sigma u\widetilde{\sigma_{*1}}- \sigma_{*1}\bar\lambda\widetilde{\sigma u}\label{e3}
\end{equation} 
for any $\sigma\in U_{2n+1}(R,\Delta)$, the last equality by Lemma \ref{38}.
\end{definition}
\begin{definition}\label{42}
Let $i,j\in\Theta_{hb}$ such that $i\neq\pm j$. Define
\begin{align*}
P_{ij}:=&e-e^{ii}-e^{jj}-e^{-i,-i}-e^{-j,-j}+e^{ij}-e^{ji}+\lambda^{(\epsilon(i)-\epsilon(j))/2}e^{-i,-j}-\lambda^{(\epsilon(j)-\epsilon(i))/2}e^{-j,-i}\\
=&T_{ij}(1)T_{ji}(-1)T_{ij}(1)\in EU_{2n+1}(R,\Delta).                                                                                                                                                                                                                                                                                                                                                                                                                                                                                             
\end{align*}
It is easy to show that $(P_{ij})^{-1}=P_{ji}$. 
\end{definition}
\begin{lemma}\label{43}
Let $i,j,k\in\Theta_{hb}$ such that $i\neq \pm j$ and $k\neq \pm i,\pm j$. Let $x\in R$ and $(y,z)\in\Delta^{-\epsilon(i)}$. Then 
\begin{enumerate}[(i)]
\item $^{P_{ki}}T_{ij}(x)=T_{kj}(x)$,
\item $^{P_{kj}}T_{ij}(x)=T_{ik}(x)$ and
\item $^{P_{-k,-i}}T_{i}(y,z)=T_{k}(y,\lambda^{(\epsilon(i)-\epsilon(k))/2}z)$.
\end{enumerate}
\end{lemma}
\begin{proof}
Straightforward.
\end{proof}
\begin{lemma}\label{new}
Let $\sigma\in U_{2n+1}(R,\Delta)$ and $i,j\in \Theta_{hb}$ such that $i\neq \pm j$. Set $\hat\sigma:={}^{P_{ij}}\sigma$. Then
\[q(\hat\sigma_{*i})=
\begin{cases}
q(\sigma_{*j}), \text{ if } \epsilon(i)=\epsilon(j),\\
q(\sigma_{*j})\+(0,-\bar{\sigma}_{ij}\sigma_{-i,j}+\overline{\bar{\sigma}_{ij}\sigma_{-i,j}}\lambda-\bar{\sigma}_{-j,j}\sigma_{jj}+\overline{\bar{\sigma}_{-j,j}\sigma_{jj}}\lambda,) \text{ if } \epsilon(i)=1,\epsilon(j)=-1,\\
q(\sigma_{*j})\+(0,-\bar{\sigma}_{-i,j}\sigma_{ij}+\overline{\bar{\sigma}_{-i,j}\sigma_{ij}}\lambda-\bar{\sigma}_{jj}\sigma_{-j,j}+\overline{\bar{\sigma}_{jj}\sigma_{-j,j}}\lambda), \text{ if } \epsilon(i)=-1, \epsilon(j)=1.
\end{cases}
\]
\end{lemma}
\begin{proof}
Straightforward computation.
\end{proof}
\subsection{Relative elementary subgroups}
\begin{definition}
Let $(I,\Omega)$ denote an odd form ideal of $(R,\Delta)$. A short root matrix $T_{ij}(x)$ is called {\it $(I,\Omega)$-elementary} if $x\in I$. An extra short root matrix $T_{i}(x,y)$ is called {\it $(I,\Omega)$-elementary} if $(x,y)\in \Omega^{-\epsilon(i)}$. If an element of $U_{2n+1}(R,\Delta)$ is an $(I,\Omega)$-elementary short or extra short root matrix, then it is called {\it $(I,\Omega)$-elementary matrix}. The subgroup $EU_{2n+1}(I,\Omega)$ of $EU_{2n+1}(R,\Delta)$ generated by the $(I,\Omega)$-elementary matrices is called the {\it preelementary subgroup of level $(I,\Omega)$}. Its normal closure $EU_{2n+1}((R,\Delta),(I,\Omega))$ in $EU_{2n+1}(R,\Delta)$ is called the {\it elementary subgroup of level $(I,\Omega)$}.
\end{definition}
\subsection{Congruence subgroups}
In this subsection $(I,\Omega)$ denotes an odd form ideal of $(R,\Delta)$. If $\sigma\in M_{2n+1}(R)$, we call the matrix $(\sigma_{ij})_{i,j\in\Theta_{hb}}\in M_{2n}(R)$ the {\it hyperbolic part of $\sigma$} and denote it by $\sigma_{hb}$. Further we define the submodule $M(R,\Delta):=\{u\in M\mid u_0\in J(\Delta)\}$ of $M$. 
\begin{definition}
The subgroup $U_{2n+1}((R,\Delta),(I,\Omega)):=$
\[\{\sigma\in U_{2n+1}(R,\Delta)\mid\sigma_{hb}\equiv e_{hb}\bmod  I\text{ and }q(\sigma u)\equiv q(u)\bmod \Omega~\forall u\in M(R,\Delta)\}\]
of $U_{2n+1}(R,\Delta)$ is called {\it the principal congruence subgroup of level $(I,\Omega)$}.
\end{definition}
\begin{lemma}\label{46}
Let $\sigma\in U_{2n+1}(R,\Delta)$. Then $\sigma\in U_{2n+1}((R,\Delta),(I,\Omega))$ iff the conditions (i) and (ii) below hold. 
\begin{enumerate}[(i)]
\item $\sigma_{hb}\equiv e_{hb}\bmod  I$.
\item $q(\sigma_{*j})\in\Omega~\forall j\in \Theta_{hb}$ and $(q(\sigma_{*0})\minus (1,0))\circ a\in\Omega~\forall a\in J(\Delta)$.
\end{enumerate}
\end{lemma}
\begin{proof}
See \cite[Lemma 28]{bak-preusser}.
\end{proof} 
\begin{remark}\label{47}
Let $\sigma\in U_{2n+1}(R,\Delta)$. Define $I_0:=\{x\in R\mid xJ(\Delta)\subseteq I\}$ and
$\tilde I_0:=\{x\in R\mid\overline{J(\Delta)}\mu x\in I_0\}$. Then $I_0$ is a left ideal of $R$ and $\tilde I_0$ an additive subgroup of $R$. If $J(\Delta)=R$ then $I_0=I$ and $\tilde I_0= \tilde I$. Lemma \ref{46} implies that $\sigma\in U_{2n+1}((R,\Delta),(I,\Omega^I_{max}))$ iff $\sigma\equiv e\bmod I,\tilde I,I_0,\tilde I_0 $, meaning that $\sigma_{hb}\equiv e_{hb}\bmod I$, $\sigma_{0j}\in\tilde I$ for any $j\in \Theta_{hb}$, $\sigma_{i0}\in I_0$ for any $i\in\Theta_{hb}$ and $\sigma_{00}-1\in\tilde I_0$.
\end{remark}
\begin{definition}
The subgroup $NU_{2n+1}((R,\Delta),(I,\Omega)):=$
\[ Normalizer_{U_{2n+1}(R,\Delta)}(U_{2n+1}((R,\Delta),(I,\Omega)))\]
of $U_{2n+1}(R,\Delta)$ is called the {\it normalized principal congruence subgroup of level $(I,\Omega)$}.
\end{definition}
\begin{remark}\label{imprem}
~\\
\vspace{-0.6cm}
\begin{enumerate}[(a)]
\item In many interesting situations, $NU_{2n+1}((R,\Delta),(I,\Omega))$ equals $U_{2n+1}(R,\Delta)$, for example the equality holds if $\Omega=\Omega^I_{min}$ or $\Omega=\Omega^I_{max}$. But it can also happen that $NU_{2n+1}((R,\Delta),(I,\Omega))\neq U_{2n+1}(R,\Delta)$, see \cite[Example 38]{bak-preusser}.
\item $EU_{2n+1}(R,\Delta)\subseteq NU_{2n+1}((R,\Delta),(I,\Omega))$, see \cite[Corollary 36]{bak-preusser}.
\end{enumerate}
\end{remark}
\begin{definition}
The subgroup $CU_{2n+1}((R,\Delta),(I,\Omega)):=$
\[ \{\sigma\in NU_{2n+1}((R,\Delta),(I,\Omega))\mid [\sigma,EU_{2n+1}(R,\Delta)]\subseteq U_{2n+1}((R,\Delta),(I,\Omega))\}\]
of $U_{2n+1}(R,\Delta)$ is called the {\it full congruence subgroup of level $(I,\Omega)$}.
\end{definition}
\begin{lemma}\label{50}
Let $\sigma\in U_{2n+1}(R,\Delta)$. Then $\sigma\in CU_{2n+1}((R,\Delta),(I,\Omega^I_{max}))$ iff
\begin{enumerate}[(i)]
\item $\sigma_{ij}\in I$ for any $i,j\in \Theta_{hb}$ such that $i\neq j$,
\item $\sigma_{i0}\in I_0$ for any $i \in \Theta_{hb}$,
\item $\sigma_{0j}\in \tilde I$ for any $j\in \Theta_{hb}$,
\item $\sigma_{ii}-\sigma_{jj}\in I$ for any $i,j\in \Theta_{hb}$ and
\item $\sigma_{00}-\sigma_{jj}\in \tilde I_0$ for any $j \in\Theta_{hb}$.
\end{enumerate}
\end{lemma}
\begin{theorem}\label{51}
If $R$ is semilocal or quasifinite and $n\geq 3$, then the equalities
\begin{align*}
&[CU_{2n+1}((R,\Delta),(I,\Omega)),EU_{2n+1}(R,\Delta)]\\
=&[EU_{2n+1}((R,\Delta),(I,\Omega)),EU_{2n+1}(R,\Delta)]\\
=&EU_{2n+1}((R,\Delta),(I,\Omega))
\end{align*}
hold. 
\end{theorem}
\begin{proof}
See \cite[Theorem 39]{bak-preusser}.
\end{proof}
\section{Sandwich classification for $U_{2n+1}(R,\Delta)$}
In this section $n$ denotes a natural number greater than or equal to $3$ and $(R,\Delta)$ a Hermitian form ring where $R$ is commutative. 
\begin{definition}
Let $\sigma\in U_{2n+1}(R,\Delta)$. Then a matrix of the form $^{\epsilon}\sigma^{\pm 1}$ where $\epsilon\in EU_{2n+1}(R,\Delta)$ is called an {\it elementary (unitary) $\sigma$-conjugate}.
\end{definition}
\begin{theorem}\label{mthm2}
Let $\sigma\in U_{2n+1}(R,\Delta)$ and $i,j,k,l\in\Theta_{hb}$ such that $k\neq\pm l$ and $i\neq \pm j$. Further let $a\in J(\Delta)$. Then 
\begin{enumerate}[(i)]
\item $T_{kl}(\sigma_{ij})$ is a product of $160$ elementary unitary $\sigma$-conjugates,
\item $T_{kl}(\sigma_{i,-i})$ is a product of $320$ elementary unitary $\sigma$-conjugates,
\item $T_{kl}(\sigma_{i0}a)$ is a product of $480$ elementary unitary $\sigma$-conjugates,
\item $T_{kl}(\bar a \mu\sigma_{0j})$ is a product of $480$ elementary unitary $\sigma$-conjugates,
\item $T_{kl}(\sigma_{ii}-\sigma_{jj})$ is a product of $480$ elementary unitary $\sigma$-conjugates,
\item $T_{kl}(\sigma_{ii}-\sigma_{-i,-i})$ is a product of $960$ elementary unitary $\sigma$-conjugates,
\item $T_{k}(q_1(\sigma_{*j}),\lambda^{-(\epsilon(k)+1)/2} q_2(\sigma_{*j}))$ is a product of $1600n+5764$ elementary unitary $\sigma$-conjugates and
\item $T_{k}((\sigma_{00}-\sigma_{jj})a,x)$ is a product of $4800n+16812$ elementary unitary $\sigma$-conjugates for some $x\in R$.
\end{enumerate}
\end{theorem}
\begin{proof}(i) Let $x\in R$. In Step 1 we show that $T_{kl}(x\bar\sigma_{23}\sigma_{21})$ is a product of $16$ elementary $\sigma$-conjugates. In Step 2 we show that $T_{kl}(x\bar\sigma_{23}\sigma_{2,-1})$ is a product of $16$ elementary $\sigma$-conjugates. In Step 3 we show that $T_{kl}(x\bar\sigma_{23}\sigma_{22})$ is a product of $32$ elementary $\sigma$-conjugates. In Step 4 we use Steps 1-3 in order to prove (i).\\
\\
{\bf Step 1.} Set $\tau:=T_{1,-2}(\bar\sigma_{23}\sigma_{23})T_{3,-2}(-\bar\sigma_{23}\sigma_{21})T_{3,-1}(\bar\lambda\bar\sigma_{23}\sigma_{22})T_{3}(0,\bar\sigma_{22}\sigma_{21}-\bar\lambda\bar\sigma_{21}\sigma_{22})$ and $\xi:={}^{\sigma}\tau^{-1}$. One checks easily that $(\sigma\tau^{-1})_{2*}=\sigma_{2*}$ and $(\tau^{-1}\sigma^{-1})_{*,-2}=\sigma'_{*,-2}$. Hence $\xi_{2*}=e^t_2$ and $\xi_{*,-2}=e_{-2}$. Set  
\begin{align*}
\zeta:={}^{\tau^{-1}}[T_{-2,-1}(1),[\tau,\sigma]]={}^{\tau^{-1}}[T_{-2,-1}(1),\tau\xi]=[\tau^{-1},T_{-2,-1}(1)][T_{-2,-1}(1),\xi],
\end{align*}
the last equality by Lemma \ref{pre}. One checks easily that $[\tau^{-1},T_{-2,-1}(1)]=T_{3,-1}(\bar\sigma_{23}\sigma_{21})T_{1}(z)$ for some $z\in \Delta^{-1}$ and $[T_{-2,-1}(1),\xi]=T_{-2}(x_{-2})\prod\limits_{i\neq 0,\pm 2 }T_{i2}(x_i)$ for some $x_i\in R~(i\neq 0,2)$. Hence \[\zeta=T_{3,-1}(\bar\sigma_{23}\sigma_{21})T_{1}(z)T_{-2}(x_{-2})\prod\limits_{i\neq 0,\pm 2 }T_{i2}(x_i).\]
It follows that $[T_{-1,3}(-x),$ $[T_{-2,3}(1),\zeta]]=T_{-2,3}(x\bar\sigma_{23}\sigma_{21})$ for any $x\in R$. Hence we have shown
\[[T_{-1,3}(-x),[T_{-2,3}(1),{}^{\tau^{-1}}[T_{-2,-1}(1),[\tau,\sigma]]]]=T_{-2,3}(x\bar\sigma_{23}\sigma_{21}).\]
This implies that $T_{-2,3}(x\bar\sigma_{23}\sigma_{21})$ is a product of $16$ elementary $\sigma$-conjugates. It follows from Lemma \ref{43} that $T_{kl}(x\bar\sigma_{23}\sigma_{21})$ is a product of $16$ elementary $\sigma$-conjugates.\\
\\
{\bf Step 2.} Step 2 can be done similarly. Namely one can show that 
\[[T_{-1,3}(-x\bar\lambda),[T_{12}(1),{}^{\tau^{-1}}[T_{-1,2}(1),[\tau,\sigma]]]]=T_{-1,2}(x\bar\sigma_{23}\sigma_{2,-1})\]
for any $x\in R$ where $\tau=T_{21}(\bar\sigma_{23}\sigma_{23})T_{31}(-\bar\sigma_{23}\sigma_{22})T_{3,-2}(\bar\sigma_{23}\sigma_{2,-1})T_{3}(0,-\bar\sigma_{22}\sigma_{2,-1}+\bar\lambda\bar\sigma_{2,-1}\sigma_{22})$.\\
\\
{\bf Step 3.} Set $\tau:=T_{21}(-\bar\sigma_{22}\sigma_{23})T_{31}(\bar\sigma_{22}\sigma_{22})T_{2,-3}(\bar\sigma_{22}\sigma_{2,-1})T_{2}(0,-\bar\sigma_{23}\sigma_{2,-1}+\bar\lambda\bar\sigma_{2,-1}\sigma_{23})$ and $\xi:={}^{\sigma}\tau^{-1}$. One checks easily that $(\sigma\tau^{-1})_{2*}=\sigma_{2*}$ and $(\tau^{-1}\sigma^{-1})_{*,-2}=\sigma'_{*,-2}$. Hence $\xi_{2*}=e^t_2$ and $\xi_{*,-2}=e_{-2}$. Set  
\begin{align*}
\zeta:={}^{\tau^{-1}}[T_{32}(1),[\tau,\sigma]]={}^{\tau^{-1}}[T_{32}(1),\tau\xi]=[\tau^{-1},T_{32}(1)][T_{32}(1),\xi].
\end{align*}
One checks easily that $\psi:=[\tau^{-1},T_{32}(1)]=T_{31}(-\bar\sigma_{22}\sigma_{23})T_{3}(y)T_{3,-2}(z)$ for some $y\in\Delta^{-1}$ and $z\in R$ and
$\theta:=[T_{32}(1),\xi]=T_{-2}(x_{-2})\prod\limits_{i\neq 0,\pm 2 }T_{i2}(x_i)$ for some $x_i\in R~(i\neq 0,2)$. Set 
\[\chi:={}^{\psi^{-1}}[T_{12}(1),\zeta]={}^{\psi^{-1}}[T_{12}(1),\psi\theta]=[\psi^{-1},T_{12}(1)][T_{12}(1),\theta].\]
One checks easily that $[\psi^{-1},T_{12}(1)]=T_{32}(\bar\sigma_{22}\sigma_{23})T_{3}(a)T_{3,-1}(b)$ for some $a\in\Delta^{-1}$ and $b\in R$ and $[T_{12}(1),\theta]=T_{-2}(c)$ for some $c\in \Delta$. Hence $\chi=T_{32}(\bar\sigma_{22}\sigma_{23})T_{3}(a)T_{3,-1}(b)T_{-2}(c)$. It follows that \[[T_{-2,3}(\bar x),[T_{2,-1}(1),\chi]]=T_{-2,-1}(-\bar x\bar\sigma_{22}\sigma_{23})\overset{(R1)}{=}T_{12}(x\bar\sigma_{23}\sigma_{22})\]
for any $x\in R$. Hence we have shown
\[[T_{-2,3}(\bar x),[T_{2,-1}(1),{}^{\psi^{-1}}[T_{12}(1),{}^{\tau^{-1}}[T_{32}(1),[\tau,\sigma]]]]]=T_{12}(x\bar\sigma_{23}\sigma_{22}).\]
This implies that $T_{12}(x\bar\sigma_{23}\sigma_{22})$ is a product of $32$ elementary $\sigma$-conjugates. It follows from Lemma \ref{43} that $T_{kl}(x\bar\sigma_{23}\sigma_{22})$ is a product of $32$ elementary $\sigma$-conjugates.\\
\\
{\bf Step 4.} Set $I:=I(\{\bar\sigma_{23}\sigma_{21},\overline{\bar\sigma_{23}\sigma_{2,-1}}\})$, $J:=I(\{\bar\sigma_{23}\sigma_{21},\overline{\bar\sigma_{23}\sigma_{2,-1}},\bar\sigma_{23}\sigma_{22}\})$ and
\[\tau:=[\sigma^{-1},T_{12}(-\bar\sigma_{23})]=(e-\sigma'_{*1}\bar\sigma_{23}\sigma_{2*}+\sigma_{*,-2}'\sigma_{23}\sigma_{-1,*})T_{12}(\bar\sigma_{23}).\]
Clearly
\[\tau_{11}=1-\sigma'_{11}\bar\sigma_{23}\sigma_{21}+\sigma_{1,-2}'\sigma_{23}\sigma_{-1,1}=1-\sigma'_{11}\underbrace{\bar\sigma_{23}\sigma_{21}}_{\in I}+\bar\lambda\underbrace{\bar\sigma_{2,-1}\sigma_{23}}_{\in I}\sigma_{-1,1},\]
the last equality by Lemma \ref{36}, and 
\[\tau_{12}=-\sigma'_{11}\bar\sigma_{23}\sigma_{22}+\sigma_{1,-2}'\sigma_{23}\sigma_{-1,2}+\tau_{11}\bar\sigma_{23}=-\sigma'_{11}\underbrace{\bar\sigma_{23}\sigma_{22}}_{\in J}+\bar\lambda\underbrace{\bar\sigma_{2,-1}\sigma_{23}}_{\in J}\sigma_{-1,2}+\tau_{11}\bar\sigma_{23},\]
again the last equality by Lemma \ref{36}.
Hence $\tau_{11}\equiv 1 \bmod I$ and $\tau_{12}\equiv \bar\sigma_{23}\bmod J$ (note that $I\subseteq J$). Set 
$\zeta:={}^{P_{13}P_{21}}\tau$. Then $\zeta_{22}=\tau_{11}$ and $\zeta_{23}=\tau_{12}$ and hence $\bar\zeta_{23}\zeta_{22}\equiv \sigma_{23}\bmod I+\bar J$. Applying Step 3 above to $\zeta$, we get that $T_{kl}(\bar\zeta_{23}\zeta_{22})$ is a product of $32$ elementary $\zeta$-conjugates. Since any elementary $\zeta$-conjugate is a product of $2$ elementary $\sigma$-conjugates, it follows that $T_{kl}(\bar\zeta_{23}\zeta_{22})$ is a product of $64$ elementary $\sigma$-conjugates. Thus, by Steps 1-3, $T_{kl}(\sigma_{23})$ is a product of $64+16+16+16+16+32=160$ elementary $\sigma$-conjugates. Since one can bring $\sigma_{ij}$ to position $(3,2)$ conjugating by monomial matrices from $EU_{2n+1}(R,\Delta)$, the assertion of (i) follows.\\
\\
(ii)-(vi) See the proof of Theorem \ref{mthm1}. In the proof of (iii) replace the matrix $T_{-j}(-1)$ by $T_{-j}(\minus_{\epsilon(j)}(a,b))$ where $b$ is chosen such that $(a,b)\in \Delta^{\epsilon(j)}$. In the proof of (iv) replace the matrix $T_{i}(-1)$ by $T_{i}(\minus_{-\epsilon(i)}(\lambda^{(\epsilon(i)-1)/2}$ $ a,b))$ where $b$ is chosen such that $(a,b)\in \Delta^{-\epsilon(i)}$.\\
\\
(vii) By Lemma \ref{43} it suffices to consider the case $\epsilon(k)=-1$. Set $m:=160$. In Step 1 we show that for any $x\in R$ the matrix $T_{k}(q(\sigma_{*1})\circ \sigma_{11}x)$ is a product of $(2n+19)m+4$ elementary $\sigma$-conjugates. In Step 2 we use Step 1 in order to prove (vii).\\
\\
{\bf Step 1.} Set $u':=\begin{pmatrix}0&\dots&0&\sigma'_{-1,-1}&-\sigma'_{-1,-2}\end{pmatrix}^t=\begin{pmatrix}0&\dots&0&\bar\sigma_{11}&-\bar\sigma_{21}\end{pmatrix}^t\in M$ and $u:=\sigma^{-1}u'\in M$. Then clearly $u_{-1}=0$. Further $q(u)=q(\sigma^{-1}u')\equiv q(u')=(0,0) \bmod \Delta$ and hence $u$ is isotropic. Set
\[
\xi:={}^{\sigma}T_{*,-1}(-u)\overset{(\ref{e3})}{=}e-\sigma u\widetilde{\sigma_{*1}}+ \sigma_{*1} \bar\lambda\widetilde{\sigma u}=e-u'\widetilde{\sigma_{*1}}+ \sigma_{*1}\bar\lambda \widetilde{u'}.
\]
Then
\begin{align*}
\xi=
\arraycolsep=8pt\def\arraystretch{1.5}\left(\begin{array}{cccccc|c|cccccc}
1-\sigma_{11}\sigma_{21}&\sigma_{11}\sigma_{11}&&&&&&&&&&&\\
-\sigma_{21}\sigma_{21}&1+\sigma_{21}\sigma_{11}&&&&&&&&&&&\\
-\sigma_{31}\sigma_{21}&\sigma_{31}\sigma_{11}&1&&&&&&&&&&\\
-\sigma_{41}\sigma_{21}&\sigma_{41}\sigma_{11}&&1&&&&&&&&\\
\vdots&\vdots&&&\ddots&&&&&&&&\\
-\sigma_{n1}\sigma_{21}&\sigma_{n1}\sigma_{11}&&&&1&&&&&&&\\
\hline-\sigma_{01}\sigma_{21}&\sigma_{01}\sigma_{11}&&&&&1&&&&&\\
\hline-\sigma_{-n,1}\sigma_{21}&\sigma_{-n,1}\sigma_{11}&&&&&&1&&&&\\
\vdots&\vdots&&&&&&&\ddots&&&&\\
-\sigma_{-4,1}\sigma_{21}&\sigma_{-4,1}\sigma_{11}&&&&&&&&1&&&\\
-\sigma_{-3,1}\sigma_{21}&\sigma_{-3,1}\sigma_{11}&&&&&&&&&1&&\\
-\lambda\bar\beta&\alpha&*&*&\dots&*&*&*&\dots&*&-\bar\sigma_{11}\bar\sigma_{31}&*&*\\
\gamma&\beta&*&*&\dots&*&*&*&\dots&*&\bar\sigma_{21}\bar\sigma_{31}&*&*
\end{array}\right)
\end{align*}
where $\alpha=\sigma_{-2,1}\sigma_{11}-\lambda\bar\sigma_{11}\bar\sigma_{-2,1}$, $\beta=\sigma_{-1,1}\sigma_{11}+\lambda\bar\sigma_{21}\bar\sigma_{-2,1}$ and $\gamma=-\sigma_{-1,1}\sigma_{21}+\lambda\bar\sigma_{21}\bar\sigma_{-1,1}$. Set \[\tau:=T_{-3,1}(\sigma_{-3,1}\sigma_{21})T_{-3,2}(-\sigma_{-3,1}\sigma_{11}).\] 
It follows from (i) that $\tau$ is a product of $2m$ elementary $\sigma$-conjugates. Clearly
\begin{align*}
\xi\tau=
\arraycolsep=8pt\def\arraystretch{1.5}\left(\begin{array}{cccccc|c|cccccc}
1-\sigma_{11}\sigma_{21}&\sigma_{11}\sigma_{11}&&&&&&&&&&&\\
-\sigma_{21}\sigma_{21}&1+\sigma_{21}\sigma_{11}&&&&&&&&&&&\\
-\sigma_{31}\sigma_{21}&\sigma_{31}\sigma_{11}&1&&&&&&&&&&\\
-\sigma_{41}\sigma_{21}&\sigma_{41}\sigma_{11}&&1&&&&&&&&\\
\vdots&\vdots&&&\ddots&&&&&&&&\\
-\sigma_{n1}\sigma_{21}&\sigma_{n1}\sigma_{11}&&&&1&&&&&&&\\
\hline-\sigma_{01}\sigma_{21}&\sigma_{01}\sigma_{11}&&&&&1&&&&&\\
\hline-\sigma_{-n,1}\sigma_{21}&\sigma_{-n,1}\sigma_{11}&&&&&&1&&&&\\
\vdots&\vdots&&&&&&&\ddots&&&&\\
-\sigma_{-4,1}\sigma_{21}&\sigma_{-4,1}\sigma_{11}&&&&&&&&1&&&\\
0&0&&&&&&&&&1&&\\
*&\alpha+\delta&0&*&\dots&*&*&*&\dots&*&*&*&*\\
*&\beta-\epsilon&0&*&\dots&*&*&*&\dots&*&*&*&*
\end{array}\right)
\end{align*}
where $\delta=\bar\sigma_{11}\bar\sigma_{31}\sigma_{-3,1}\sigma_{11}$ and $\epsilon=\bar\sigma_{21}\bar\sigma_{31}\sigma_{-3,1}\sigma_{11}$. Let $x\in R$ and set  
\begin{align*}
\zeta:={}&^{T_{*,-1}(-u)}[T_{2,-3}(-x),[T_{*,-1}(u),\sigma]\tau]\\
={}&^{T_{*,-1}(-u)}[T_{2,-3}(-x),T_{*,-1}(u)\xi\tau]
\\=&[T_{*,-1}(-u),T_{2,-3}(-x)][T_{2,-3}(-u),\xi\tau].
\end{align*}
Clearly $\zeta$ is a product of $4m+4$ elementary $\sigma$-conjugates. One checks easily that 
\begin{align*}
&[T_{*,-1}(-u),T_{2,-3}(-x)]\\
=&T_{1}(0,\bar\lambda(
-\bar xu_{-2}\bar u_{-3}+\lambda\overline{\bar xu_{-2}\bar u_{-3}}))T_{1,-2}(\bar\lambda\bar x \bar u_{-3})T_{1,-3}(-x\bar u_{-2})\\
=&T_{1}(0,\bar\lambda(a+\lambda\bar a))T_{1,-2}(b)T_{1,-3}(-x(\sigma_{11}\sigma_{22}-\sigma_{12}\sigma_{21}))
\end{align*}
for some $a,b\in I(\{\sigma_{21},\sigma_{23}\})$. Further
\[[T_{2,-3}(-x),\xi\tau]=(\prod\limits_{\substack{p=1,\\p\neq 3,0}}^{-4}T_{p,-3}(x\sigma_{p1}\sigma_{11}))T_{-2,-3}(x(\alpha+\delta))T_{-1,-3}(x(\beta-\epsilon))T_{3}(y)\]
where $y=(q_1(\sigma_{*1}),\bar\lambda q_2(\sigma_{*1}))\circ \sigma_{11}x+(0,c-\bar\lambda\bar c)$ for some $c\in I(\{\sigma_{31},\sigma_{-2,1}\})$. Hence
\begin{align*}
\zeta=&T_{1}(0,\bar\lambda(a+\lambda\bar a))T_{1,-2}(b)T_{1,-3}(-x(\sigma_{11}(\sigma_{22}-\sigma_{11})-\sigma_{12}\sigma_{21}))\cdot\\
&\cdot (\prod\limits_{\substack{p=2,\\p\neq 3,0}}^{-4}T_{p,-3}(x\sigma_{p1}\sigma_{11}))T_{-2,-3}(x(\alpha+\delta))T_{-1,-3}(x(\beta-\epsilon))T_{3}(y).
\end{align*}
It follows from (i), (ii) and (iii) and relation (S5) in Lemma \ref{39} that $T_{3}(y)$ is a product of $4m+4+4m+2m+4m+(2n-5)m+3m+3m=(2n+15)m+4$ elementary $\sigma$-conjugates. By (i) and relation (S5) in Lemma \ref{39}, $T_{3}(0,-c+\bar\lambda\bar c)$ is a product of $4m$ elementary $\sigma$-conjugates. Hence $T_{3}((q_1(\sigma_{*1}),\bar\lambda q_2(\sigma_{*1}))\circ \sigma_{11}x)=T_{3}(y)T_{3}(0,-c+\bar\lambda\bar c)$ is a product of $(2n+19)m+4$ elementary $\sigma$-conjugates. It follows from Lemma \ref{43} that $T_{k}(q(\sigma_{*1})\circ \sigma_{11}x)$ is a product of $(2n+19)m+4$ elementary $\sigma$-conjugates.\\
\\
{\bf Step 2.} By Lemma \ref{last} we have
\begin{align*}
&T_{k}(q(\sigma_{*1}))\\
=&T_{k}(q(\sigma_{*1})\circ\sum\limits_{s=1}^{-1}\sigma'_{1s}\sigma_{s1})\\
=&T_{k}((\plus\limits_{s=1}^{-1} (q(\sigma_{*1})\circ\sigma'_{1s}\sigma_{s1})\+(0,\sum\limits_{\substack{s,t=1,\\s>t}}^{-1} \sigma'_{1s}\sigma_{s1}q_2(\sigma_{*1})\overline{\sigma'_{1t}\sigma_{t1}}-\overline{\sigma'_{1s}\sigma_{s1}q_2(\sigma_{*1})\overline{\sigma'_{1t}\sigma_{t1}}}\lambda))\\
=&\underbrace{\prod\limits_{s=1}^{-1}T_k(q(\sigma_{*1})\circ\sigma'_{1s}\sigma_{s1})}_{A:=}\underbrace{T_k(0,\sum\limits_{\substack{s,t=1,\\s>t}}^{-1} \sigma'_{1s}\sigma_{s1}q_2(\sigma_{*1})\overline{\sigma'_{1t}\sigma_{t1}}-\overline{\sigma'_{1s}\sigma_{s1}q_2(\sigma_{*1})\overline{\sigma'_{1t}\sigma_{t1}}}\lambda)}_{B:=}
\end{align*}
since $q(\sigma_{*1})\in \Delta\subseteq \Delta_{max}$. By Step 1, $T_{k}(q(\sigma_{*1})\circ \sigma'_{11}\sigma_{11})$ is a product of $(2n+19)m+4$ elementary $\sigma$-conjugates. By (i) and relation (SE2) in Lemma \ref{39}, $T_k(q(\sigma_{*1})\circ\sigma'_{1s}\sigma_{s1})$ is a product of $3m$ elementary $\sigma$-conjugates if $s\neq \pm 1,0$. By (ii) and relation (SE2), $T_k(q(\sigma_{*1})\circ\sigma'_{1,-1}\sigma_{-1,1})$ is a product of $6m$ elementary $\sigma$-conjugates. By (iv) and relation (SE2), $T_k(q(\sigma_{*1})\circ\sigma'_{10}\sigma_{01})$ is a product of $9m$ elementary $\sigma$-conjugates (note that $\sigma'_{10}=\bar\lambda\bar\sigma_{0,-1}\mu$ by Lemma \ref{36}). Hence $A$ is a product of $(2n+19)m+4+(2n-2)\cdot 3m+6m+9m=(8n+28)m+4$ elementary $\sigma$-conjugates. On the other hand $B=T_{k}(0,x-\lambda\bar x)$ for some $x\in I(q_2(\sigma_{*1}))$. Since $q_2(\sigma_{*1})=\sum\limits_{r=1}^n\bar\sigma_{r1}\sigma_{-r,1}$, it follows from (i), (ii) and relation (S5) in Lemma \ref{39} that $B$ is a product of $4m+(n-1)\cdot 2m=(2n+2)m$ elementary $\sigma$-conjugates. Hence $T_{k}(q(\sigma_{*1}))$ is a product of $(10n+30)m+4=1600n+4804$ elementary $\sigma$-conjugates. The assertion of (vii) follows now from Lemma \ref{new}.\\
\\
(viii) By Lemma \ref{43} it suffices to consider the case $\epsilon(k)=-1$. Choose a $b\in R$ such that $(a,b)\in \Delta^{\epsilon(j)}$. One checks easily that the entry of $\xi:={}^{T_{-j}(a,b)}\sigma$ at position $(0,j)$ equals $(\sigma_{jj}-\sigma_{00})a+\sigma_{0j}-\sigma_{j0}a^2+\sigma_{0,-j}b+\sigma_{j,-j}ba^2$. By applying (vii) to $\xi$ we get that $T_{k}(q(\xi_{*j}))=T_{k}((\sigma_{jj}-\sigma_{00})a+\sigma_{0j}-\sigma_{j0}a^2+\sigma_{0,-j}b+\sigma_{j,-j}ba^2, x_1)$ is a product of $1600n+4804$ elementary $\sigma$-conjugates where $x_1$ is some ring element. By applying (vii) to $\sigma$ we get that $T_{k}(\minus q(\sigma_{*j}))=T_k(-\sigma_{0j},x_2)$ and $T_{k}(\minus q(\sigma_{*,-j})\circ b)=T_{k}(-\sigma_{0,-j}b,x_3)$ each are a product of $1600n+4804$ elementary $\sigma$-conjugates where $x_2,x_3$ are some ring elements. By (iii) and relation (SE2) in Lemma \ref{39}, $T_{k}(\sigma_{j0}a^2,x_4)$ is a product of $3\cdot 480=1440$ elementary $\sigma$-conjugates where $x_4$ is some ring element. By (ii) and relation (SE2) in Lemma \ref{39}, $T_{k}(-\sigma_{j,-j}ba^2,x_5)$ is a product of $3\cdot 320=960$ elementary $\sigma$-conjugates where $x_5$ is some ring element. It follows that $=T_{k}((\sigma_{jj}-\sigma_{00})a+\sigma_{0j}-\sigma_{j0}a^2+\sigma_{0,-j}b+\sigma_{j,-j}ba^2,x_1)T_k(-\sigma_{0j},x_2)T_{k}(\sigma_{j0}a^2,x_4)T_{k}(-\sigma_{0,-j}b,x_3)T_{k}(-\sigma_{j,-j}ba^2,x_5)=T_{k}((\sigma_{jj}-\sigma_{00})a,x)$ is a product of $3\cdot(1600n+4804)+1440+960=4800n+16812$ elementary $\sigma$-conjugates where $x$ is some ring element.
\end{proof}

As a corollary we get the Sandwich Classification Theorem for $U_{2n+1}(R,\Delta)$.
\begin{corollary}
Let $H$ be a subgroup of $U_{2n+1}(R,\Delta)$. Then $H$ is normalized by $EU_{2n+1}(R,\Delta)$ iff 
\begin{equation}
EU_{2n+1}((R,\Delta),(I,\Omega))\subseteq H\subseteq CU_{2n+1}((R,\Delta),(I,\Omega))
\end{equation}
for some odd form ideal $(I,\Omega)$ of $(R,\Delta)$.
\end{corollary}
\begin{proof}
First suppose that $H$ is is normalized by $EU_{2n+1}(R,\Delta)$. Let $(I,\Omega)$ be the odd form ideal of $(R,\Delta)$ defined by $I:=\{x\in R\mid T_{12}(x)\in H\}$ and $\Omega:=\{(y,z)\in \Delta\mid T_{-1}(y,z)\in H\}$. Then clearly $EU_{2n+1}((R,\Delta),(I,\Omega))\subseteq H$. It remains to show that $H\subseteq CU_{2n+1}((R,\Delta),(I,\Omega))$, i.e. that if $\sigma\in H$ and $\epsilon\in EU_{2n+1}(R,\Delta)$, then $[\sigma,\epsilon]\in U_{2n+1}((R,\Delta),(I,\Omega))$. Since $U_{2n+1}((R,\Delta),(I,\Omega))$ is normalized by $EU_{2n+1}(R,\Delta)$ (see Remark \ref{imprem} (b)), we can assume that $\epsilon$ is an elementary matrix. By the previous theorem and Lemma \ref{50}, we have $\sigma\in CU_{2n+1}((R,\Delta),(I,\Omega^I_{max}))$ and therefore $[\sigma,\epsilon]\in U_{2n+1}((R,\Delta),(I,\Omega_{max}^I))$. Hence, by Lemma \ref{46} it remains to show that $q([\sigma,\epsilon]_{*j})\in\Omega~\forall j\in \Theta_{hb}$ and $(q([\sigma,\epsilon]_{*0})\minus (1,0))\circ a\in\Omega~\forall a\in J(\Delta)$. But applying the previous theorem to $[\sigma,\epsilon]$ we get that $q([\sigma,\epsilon]_{*j})\in\Omega~\forall j\in \Theta_{hb}$. That $(q([\sigma,\epsilon]_{*0})\minus (1,0))\circ a\in\Omega~\forall a\in J(\Delta)$ follows from the previous theorem and \cite[Lemma 63]{bak-preusser}. Suppose now that (4) holds for some odd form ideal $(I,\Omega)$. Then it follows from the standard commutator formula in Theorem \ref{51} that $H$ is normalized by $EU_{2n+1}(R,\Delta)$.
\end{proof}

\end{document}